\documentclass[12pt]{article}
\usepackage{algorithm, algpseudocode}
\usepackage{amsmath,amssymb,amsthm}
\usepackage{mathrsfs}
\usepackage{undertilde} 
\usepackage{subfig}
\usepackage{color} 
\usepackage{graphicx}
\usepackage{makecell}
\usepackage{tabularx}
\usepackage{multirow}
\usepackage{multicol}
\usepackage{arydshln}
\usepackage{epstopdf}
\usepackage{enumerate}
\usepackage{overpic}
\usepackage{diagbox}
\usepackage{unicode}
 \vspace{0.6cm}

\newtheorem{lemma}{Lemma}[section]

\newtheorem{them}[lemma]{Theorem}

\newtheorem{rem}{Remark}

\newtheorem{que}{Question}
\newtheorem{exam}{Example}

\begin{document}

\title{{Constructing cospectral graphs via regular rational orthogonal matrix with level two and three}\footnote{This work is supported by the National Natural Science Foundation of China (Grant No. 11971376)}}

\author{\small $^a$Lihuan Mao\footnote{The corresponding author. E-mail address: maolihuan521@163.com} \qquad $^a$Fu Yan \\
\small $^a$ School of Mathematics and Data Science, Shaanxi University of Science and Technology,\\
\small Xi'an Weiyang University Park, Xi'an, P.R. China,
710021}
\date{}
 \maketitle

\abstract

Two graphs $G$ and $H$ are \emph{cospectral} if the adjacency matrices share the same spectrum. Constructing cospectral non-isomorphic graphs has been studied extensively for many years and various constructions are known in the literature, e.g. the famous GM-switching method. In this paper, we shall construct cospectral graphs via regular rational orthogonal matrix $Q$ with level two and three. We provide two straightforward algorithms to characterize with adjacency matrix $A$ of graph $G$ such that $Q^TAQ$ is again a (0,1)-matrix, and introduce two new switching methods to construct families of cospectral graphs which generalized the GM-switching to some extent.


\section{Introduction}

All graphs considered in this paper are undirected, finite and simple graphs. For some
notations and terminologies in graph spectra, see~\cite{CDS}.

Let $G=(V,E)$ be a graph with vertex set $V(G)=\{v_1,v_2,\ldots,v_n\}$ and edge set $E(G)=\{e_1,e_2,\ldots,e_m\}$. Let $A(G)$ be the (0,1)-adjacency matrix of $G$, \emph{the characteristic polynomial} of $G$ is defined as $P_G(\lambda)=\det(\lambda I-A(G))$. The spectrum of $G$ consists of all the eigenvalues of $G$, including their respective multiplicities. Two graphs $G$ and $H$ are \emph{cospectral} if they share the same adjacency spectrum.

For a given graph $G$, we say that $G$ is determined by its spectrum ($DS$ for short), if any graph having the same spectrum as $G$ is necessarily isomorphic to $G$ (of course, the spectrum concerned should be specified). ``Which graphs are determined by their spectrum?''  has been a long-standing unsolved, fundamental problem in spectral graph theory and it has attracted much attention of
researchers in recent years. There are two areas of research on this issue. The first area is to prove graphs are determined by their spectrum, e.g., starlike trees \cite{GRO}, lollipop graph \cite{WHH}, $\infty$-graph \cite{JFW}, etc. Actually a lot of results for $DS$-graphs are mostly limited to special graph classes. The other area is to find methods to construct cospectral non-isomorphic graphs, such as Seidel switching~\cite{JHV-JJS}, Schwenk's construction~\cite{AJS}, Sunada's method~\cite{SUN}, and among which, the GM-switching, invented by Godsil and McKay~\cite{CDG-BDM}, is proved to a simple and powerful one. Some of the most recent results about constructions on cospectral graphs see (\cite{AA-WHH},\cite{YZJ-SCG-WW},\cite{LHM-WW},\cite{LHQ-YZJ-WW},\cite{WW-LHQ-YLH}). So many cospectral non-isomorphic examples, to some extent, can reflect the complexity of the $DS$-problem.

We know that the proof of GM-switching is straightforward with a regular rational orthogonal matrix $Q=diag(Q_k,I_{n-k})$ (where $Q_k=\frac{1}{k}J-I_k$). In \cite{AA-WHH}, Abiad and Haemers considered the problem of replacing the above $Q_k$ with another rational orthogonal matrix of level two and small order $k$ ($k=6,7,8$). The \textit{level} of a rational orthogonal matrix $Q$ with $Qe=e$
is the smallest positive integer $\ell$ such that $\ell Q$ is an integral
matrix.  In \cite{WW-LHQ-YLH}, Wang et al.  proposed to use another class of rational orthogonal matrices with level $p$ (an odd prime) for the construction. In \cite{LHM-WW}, Mao et al. focus on the special case that above $Q_k$ is a fully indecomposable regular rational orthogonal matrix with level $\ell = 2$ (see Theorem 2.1 case(i) for details). They also give several problems that have not been completely solved below.

\begin{que}
\rm{\cite{LHM-WW} Given a rational orthogonal matrix $Q$ with level $\ell$, how to characterize graph $G$ with adjacency matrix $A$ such that $Q^TAQ$ is again a $(0,1)$-matrix?}
\end{que}
\begin{que}
\rm{\cite{LHM-WW} For larger $m$, we find a family of generalized sun-graph $S_{2k}$ for which there exists a $Q$ with level two such that
$Q^TA(S_{2k})Q$ is a (0,1)-matrix. What about other graph families? Can something else be said in this direction?}
\end{que}
\begin{que}
\rm{\cite{LHM-WW} Can the method in this paper be applied to deal with the case that $Q$ is a regular rational orthogonal matrix
of level two with several fully indecomposable blocks? More generally, can one say something to Question 1 for those $Q$'s
with level $\ell > 2$?}
\end{que}

 A full answer to these three questions  seems out of reach at present. In this paper, we will give partial answers to them and summarize the main contributions of the paper as follows:

\begin{enumerate}
  \item Let $A$ be the adjacency matrix of a graph $G$. We give two algorithms to construct $A$  by selecting some given matrices combination such that $Q_k^TAQ_k$ is again a (0,1)-matrix, where $Q_k$ is a special case with level $\ell=2$ and $\ell=3$ (see Theorem 2.1). And several families of infinite matrix $A$ are completely characterized (see Example 2, Example 3, Example 4, Example 5).
  \item Based on the above results, we give two new switching methods for constructing cospectral but non-isomorphic graphs with regular rational orthogonal matrix $Q$ of the form  $Q = {\rm diag}(Q_k, I_{n-k}) (k=2m~{\rm or}~3m)$, which generalized the GM-switching to some extent.
\end{enumerate}

The rest of the paper is organized as follows: In Section 2, we give some preliminary
results that are needed in the following sections. In Section 3, we give two algorithms to construct adjacency matrix $A$ such that $Q_k^TAQ_k$ is a (0,1)-matrix, where $Q_k$ has one nontrivial indecomposable block of order $2m$ (resp. $3m$) of level 2 (resp. 3). In Section 4, we give two new methods of constructing cospectral graph. Conclusions are given in Section 5.

\section{Preliminaries}

For convenience of the reader, in this section, we will briefly
review some known results from \cite{WX1,WW} that will be needed in the sequel.

Define
$$\mathcal{Q}(G)=\left\{ \begin{array}{rrr}Q~{\rm is ~a~rational}~~~&\vline&Qe=e,~Q^TA(G)Q ~{\rm is~a~ symmetric}~\\
   ~{\rm orthogonal} ~{\rm matrix}&\vline&   {\rm (0,1)-matrix~ with~ zero~ diagonal}
  \end{array}\right\},$$

  where $e$ is the all-one vector.

 A matrix $M_{n\times n}$ is called \textit{partly
decomposable}, if there exists permutation matrices $P_1$ and $P_2$
such that
$${{P_1MP_2= \left[\begin{array}{cc}
  M_1 & M_2 \\
  O & M_3
\end{array}\right]}},$$
where $M_1$ and $M_3$ are square matrices. If a matrix is not
partly decomposable, it is called \textit{fully indecomposable}
(see \cite{BR}). If an orthogonal matrix $Q$ is partly
decomposable, then it can be converted into the following form by
permuting rows and permuting columns: $${\left[\begin{array}{cc}
  Q_1 & O \\
  O & Q_2
\end{array}\right]},$$

where $Q_1$ and $Q_2$ are orthogonal matrices.

The following theorem gives the basic structures of the fully indecomposable rational orthogonal matrices with level $\ell = 2$ or  $\ell = 3$ and $Qe = e$.
\begin{them}[\cite{WX1}]\label{Q1}
Let $Q$ be a fully indecomposable rational orthogonal matrix
with level $\ell = 2$ or  $\ell = 3$ and $Qe=e$. Then there exist permutation
matrices $P_1$ and $P_2$ such that $P_1QP_2$ can be written as one of the following canonical form:\\
(i)~{\scriptsize{ $\frac{1}{2}\left[\begin{array}{rrrrrr}
Y_1&O_1&\cdots&\cdots&O_1&J_1\\
J_1&Y_1&O_1&\cdots&\cdots&O_1\\
O_1&J_1&Y_1& O_1&\cdots&O_1\\
\vdots& \ddots& \ddots& \ddots& \ddots&\vdots\\
O_1&\cdots&O_1&J_1&Y_1&O_1\\
 O_1&\cdots&\cdots & O_1& J_1& Y_1
\end{array}\right]$ }},
(ii)~{\scriptsize{$\frac{1}{3}\left[\begin{array}{rrrrrr}
Y_2&O_2&\cdots&\cdots&O_2&J_2\\
J_2&Y_2&O_2&\cdots&\cdots&O_2\\
O_2&J_2&Y_2& O_2&\cdots&O_2\\
\vdots& \ddots& \ddots& \ddots& \ddots&\vdots\\
O_2&\cdots&O_2&J_2&Y_2&O_2\\
 O_2&\cdots&\cdots & O_2& J_2& Y_2
\end{array}\right]$}}. \\
\end{them}

where $O_1,I_1,J_1,Y_1=J_1-2I_1$ are zero matrix, Identity matrix, all-one matrix and square matrix of order 2. $O_2,I_2,J_2,  Y_2=J_2-3I_2$ are zero matrix, Identity matrix, all-one matrix and square matrix of order 3.

Let $Q$ be an arbitrary rational orthogonal matrix with level two or three
and $Qe=e$. Then $Q$ can be written as the following canonical
form by permuting rows and columns appropriately: $diag(Q_1,Q_2)$,
where $Q_i$'s are matrices in Theorem \ref{Q1} or an identity matrix.

\section{Characterization of adjacency matrix $A$ for which $Q_k^TAQ_k$ is a (0,1)-matrix}

Let $Q_k$ be a $k$ by $k$ fully indecomposable regular rational orthogonal matrix with level two or three, as given in Theorem~\ref{Q1}. In this section, we mainly provide two straightforward algorithms for constructing adjacency matrix $A$, such that $Q_k^TAQ_k$ is again a (0,1)-matrix, which gives partial answer to Question 1.

We divide the matrix $A$ into blocks, and denote by $A_{i,j}$ the $(i,j)$-th block of $A$.

If the regular rational orthogonal matrix $Q_k$ with $\ell=2$ and $k=2m$, then divide the matrix $A$ into  $m$ by $m$ blocks.

If the regular rational orthogonal matrix $Q_k$ with $\ell=3$ and $k=3m$, then divide the matrix $A$ into  $m$ by $m$ blocks.

Since the adjacency matrix is symmetric, in the following two algorithms we only consider the upper triangular part of the adjacency matrix $A$, generated from the top-left to the bottom-right along the diagonal.  For example, the order in which sub-matrix blocks are generated for 6 by 6 block $A$ is (see Fig. 1):

Step 1 generated the first diagonal block (Yellow): $A_{1,1},A_{2,2},A_{3,3},A_{4,4},A_{5,5},A_{6,6}$,

Step 2 generated the second diagonal block (Red):$A_{1,2},A_{2,3},A_{3,4},A_{4,5},A_{5,6}$,

Step 3 generated the third diagonal block (Green):$A_{1,3},A_{2,4},A_{3,5},A_{4,6}$,

Step 4 generated the forth diagonal block (Blue):$A_{1,4},A_{2,5},A_{3,6}$,

Step 5 generated  the fifth diagonal block (Magenta):$A_{1,5},A_{2,6}$

Step 6 generated the sixth diagonal block(Cyan):$A_{1,6}$.

After generation, we reassign $A_{i,i}$ with $\frac{1}{2}A_{i,i}$ (for all $i$) , then reassign $A$ with $A+A^T$, finally we obtain the complete matrix A.

\begin{figure}[htbp]
\centering
\subfloat{
\begin{minipage}[t]{0.15\textwidth}
   \centering
   \includegraphics[width=2cm,height=2cm]{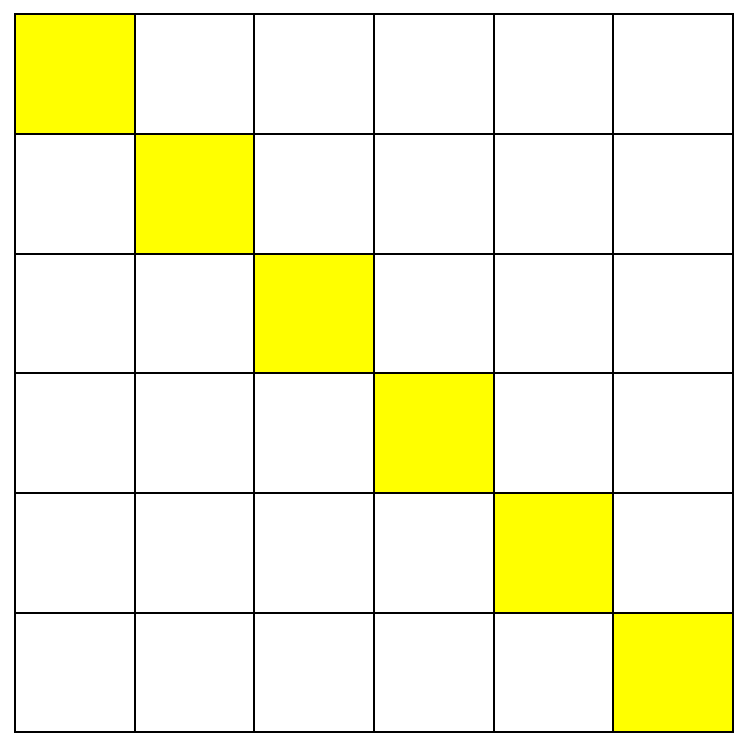}
\end{minipage}
}
\subfloat{
\begin{minipage}[t]{0.15\textwidth}
   \centering
   \includegraphics[width=2cm,height=2cm]{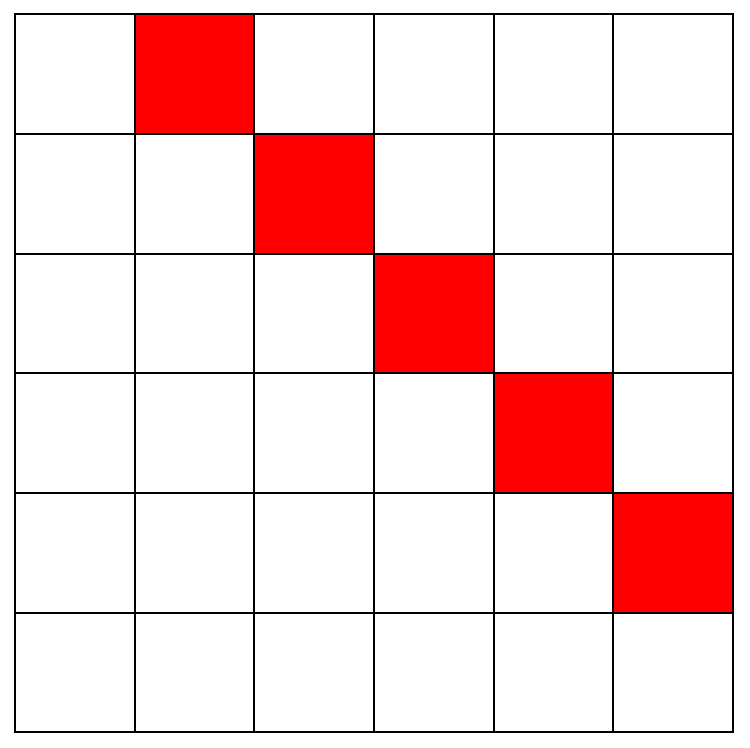}
\end{minipage}
}
\subfloat{
\begin{minipage}[t]{0.15\textwidth}
   \centering
   \includegraphics[width=2cm,height=2cm]{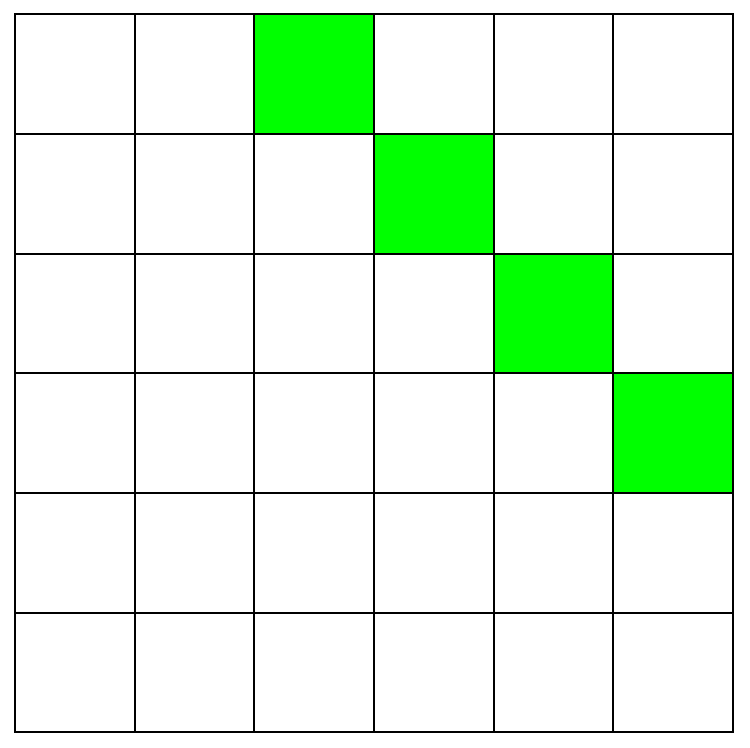}
\end{minipage}
}
\subfloat{
\begin{minipage}[t]{0.15\textwidth}
   \centering
   \includegraphics[width=2cm,height=2cm]{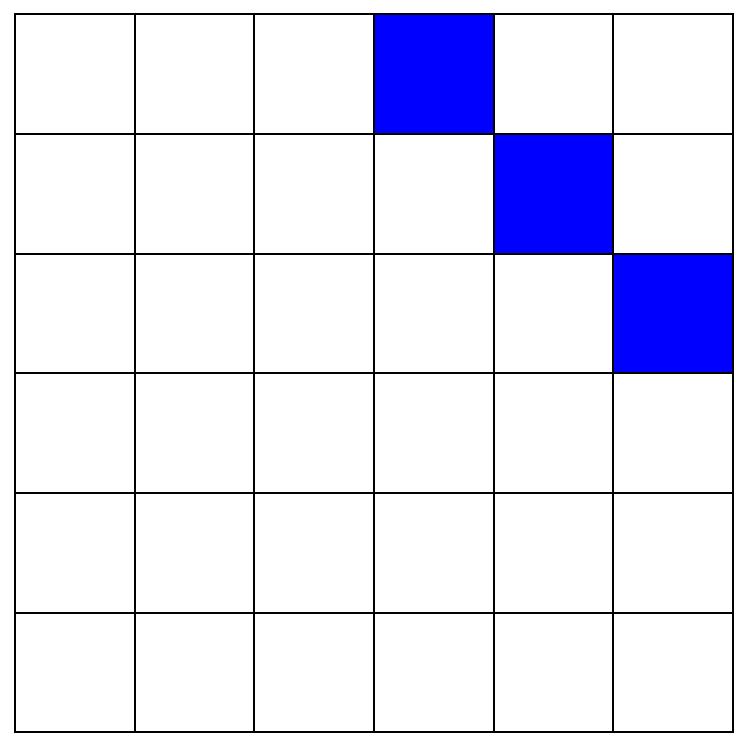}
\end{minipage}
}
\subfloat{
\begin{minipage}[t]{0.15\textwidth}
   \centering
   \includegraphics[width=2cm,height=2cm]{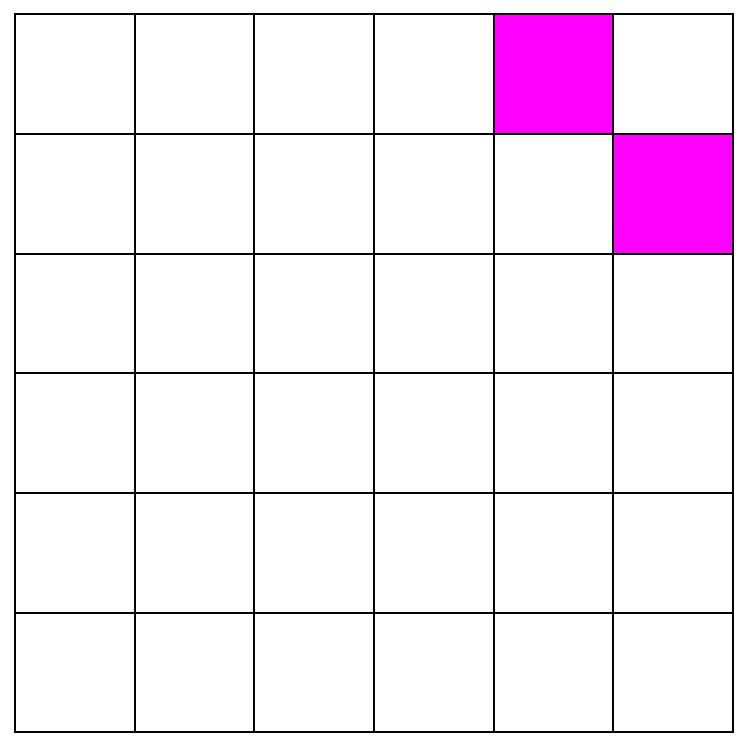}
\end{minipage}
}
\subfloat{
\begin{minipage}[t]{0.15\textwidth}
   \centering
   \includegraphics[width=2cm,height=2cm]{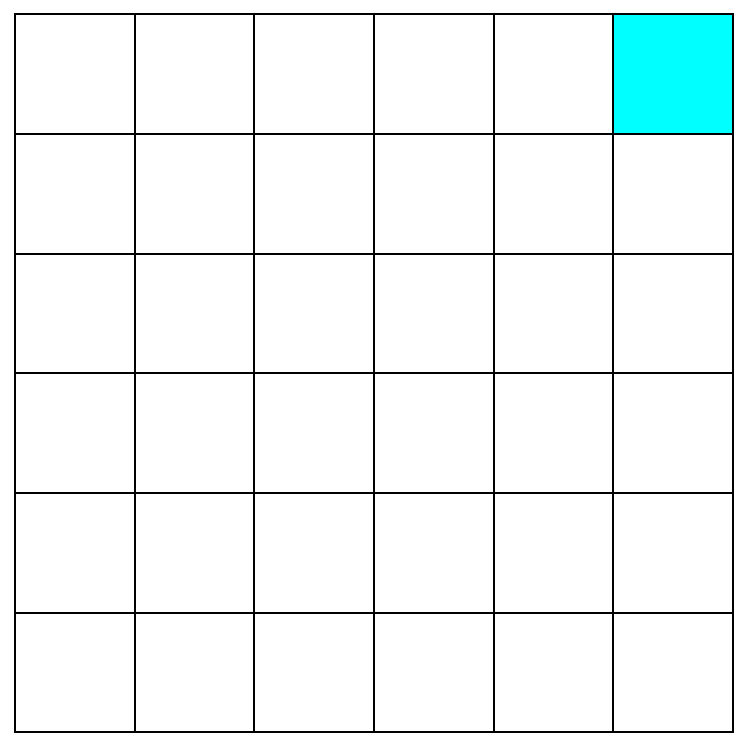}
\end{minipage}
}
\caption{Step generated order of adjacency matrix $A$}
\end{figure}

\noindent \textbf{3.1. $Q_k$ has the form in Theorem 2.1 case (i)}

For 2 by 2  $(0,1)$-matrix, there are 16 cases in total below.
$$A_1=\left[\begin{array}{cc}
0& 0\\
0& 0
\end{array}\right],A_2=\left[\begin{array}{cc}
1& 0\\
1& 0
\end{array}\right],A_3=\left[\begin{array}{cc}
1& 0\\
0& 1
\end{array}\right],A_4=\left[\begin{array}{cc}
1& 1\\
0& 0
\end{array}\right],$$
$$A_5=\left[\begin{array}{cc}
1& 1\\
1& 1
\end{array}\right],A_6=\left[\begin{array}{cc}
0& 1\\
0& 1
\end{array}\right],A_7=\left[\begin{array}{cc}
0&1\\
1&0
\end{array}\right],A_8=\left[\begin{array}{cc}
0& 0\\
1& 1
\end{array}\right],$$
$$A_9=\left[\begin{array}{cc}
1& 0\\
0& 0
\end{array}\right],A_{10}=\left[\begin{array}{cc}
0& 1\\
0& 0
\end{array}\right],A_{11}=\left[\begin{array}{cc}
0&0\\
1&0
\end{array}\right],A_{12}=\left[\begin{array}{cc}
0& 0\\
0& 1
\end{array}\right],$$
$$A_{13}=\left[\begin{array}{cc}
1& 1\\
1& 0
\end{array}\right],A_{14}=\left[\begin{array}{cc}
1& 0\\
1& 1
\end{array}\right],A_{15}=\left[\begin{array}{cc}
1&1\\
0&1
\end{array}\right],A_{16}=\left[\begin{array}{cc}
0& 1\\
1& 1
\end{array}\right].$$

Let $X_1=\left[\begin{array}{c}
J_1-2I_1\\
J_1
\end{array}\right]$ be a $2\times 1$ block matrix and $M=\left[\begin{array}{cc}
A_{i,j}&A_{i,j+1}\\
A_{i+1,j}&A_{i+1,j+1}
\end{array}\right]$ for $i=1,2,\cdots,m$ and $i\leq j$ (reducing mod $m$ if the index is larger than $m$) be a $2\times 2$ block matrix, where each sub-matrix is chosen from one of the matrix $A_t$ $(t=1,2,\ldots,16)$. By block matrix multiplication operation we have the $(i,i+1)$-th row and $(j,j+1)$-th column of  $Q_k^TAQ_k$ is $\frac{1}{4}X_1^TMX_1$. So $Q_k^TAQ_k$ is again (0,1)-matrix if only if $\frac{1}{4}X_1^TMX_1$ is a (0,1)-matrix.

The following two theorems play a key role in the algorithm which is used to construct adjacency matrix $A$ such that $Q_k^TAQ_k$ is a (0,1)-matrix.

\begin{them} \label{Q3}
Let $X_1=\left[\begin{array}{c}
J_1-2I_1\\
J_1
\end{array}\right]$ be a $2\times 1$ block matrix, $M=\left[\begin{array}{cc}
A_{i,j}&A_{i,j+1}\\
A_{i+1,j}&A_{i+1,j+1}
\end{array}\right]$, then $\frac{1}{4}X_1^TMX_1$ is a (0,1)-matrix if and only if either all these four sub-matrix $A_{i,j},A_{i,j+1},A_{i+1,j},A_{i+1,j+1}$ can only be chosen from $A_t$ $(t=1,2,\ldots,8)$ or  they can only be chosen from $A_s$ $(s=10,11,\ldots,16)$.
\end{them}

\begin{proof}
 We know that for each sub-matrix $A_{i,j},A_{i,j+1},A_{i+1,j},A_{i+1,j+1}$ it can be chosen from $A_t (t=1,2,\cdots,16)$. So there are 65536 possibilities in total such that $\frac{1}{4}X_1^TMX_1$ is a (0,1)-matrix. All possibilities can be verified directly by computer. Finally, the result of the calculation shows that the conclusion of the theorem holds.
\end{proof}

\begin{them} \label{Q4}
Let $X_1=\left[\begin{array}{c}
J_1-2I_1\\
J_1
\end{array}\right]$ be a $2\times 1$ block matrix, $M=\left[\begin{array}{cc}
A_{i,i}&A_{i,i+1}\\
A_{i,i+1}^T&A_{i+1,i+1}
\end{array}\right]$, then $\frac{1}{4}X_1^TMX_1$ is a (0,1)-matrix if and only if the first diagonal block $A_{1,1}=A_{2,2}=\cdots=A_{m,m}=A_1$ or $A_{1,1}=A_{2,2}=\cdots=A_{m,m}=A_7$. The second diagonal block  $A_{i,i+1}$ can only be chosen from one of these six matrices $A_t (t=1,2,3,5,6,7)$.
\end{them}

\begin{proof}

For the first diagonal block $A_{i,i}$ and $A_{i+1,i+1}$ there are two possibilities which they can be chosen from $A_1$ or $A_7$. For the second diagonal block $A_{i,i+1}$, By theorem \ref{Q3} there are 8 possibilities which it can be chosen from $A_t (t=1,2,\cdots,8)$. So there are 32 possibilities in total such that $\frac{1}{4}X_1^TMX_1$ is a (0,1)-matrix.  All possibilities can be verified directly by computer and finally we found that $\frac{1}{4}X_1^TMX_1$ is a (0,1)-matrix if and only if $A_{1,1}=A_{2,2}=\cdots=A_{m,m}=A_1$ or $A_{1,1}=A_{2,2}=\cdots=A_{m,m}=A_7$. In addition, $A_{i,i+1}$ can only be chosen from one of these six matrices $A_t (t=1,2,3,5,6,7)$. We complete the proof.
\end{proof}

By theorem \ref{Q3} and theorem \ref{Q4}, only $A_t$ $(t=1,2,\ldots,8)$ are useful for constructing the adjacency matrix $A$ and the other eight matrices $A_s$ $(s=10,11,\ldots,16)$ are not satisfied. Partition the $2m\times 2m$ adjacency matrix $A$ into 2 by 2 blocks and denote by $A_{i,j}$ the $(i,j)$-th block of $A$. The adjacency matrices of graphs such that $Q_k^TAQ_k$ is a (0,1)-matrix are generated as follows:

\noindent

\textbf{Algorithm 1:}

\begin{itemize}
  \item[\textbf{Step 1.}] Input parameters $m$, matrix $A_i(i=1,2,\cdots,8)$, regular rational orthogonal matrix $Q_{2m}$, $A\leftarrow 0$.
  \item[\textbf{Step 2.}] For $i=1$ to $m$, $A_{i,i}\leftarrow A_1$ (for all $i$) or $A_{i,i}\leftarrow A_7$ (for all $i$).
  \item[\textbf{Step 3.}] For $i=1$ to $m-1$, $A_{i,i+1}$ can choose any one of the six matrices $A_i (i=1,2,3,5,6,7)$.
  \item[\textbf{Step 4.}] For $k=1$ to $m-2$,

for $i=m-1-k$ to $1$,

\vspace{0.1cm}

If submatrix $\{A_{i,i+k},A_{i+1,i+k},A_{i+1,i+k+1}\}=\{A_p,A_q,A_r\}(p=1,2,4,5,6,8;q=1,3,4,5,7,8;r=2,3,4,6,7,8)$, then $A_{i,i+k+1}$ can be any one matrix of $A_4,A_8$.

\vspace{0.1cm}

If submatrix $\{A_{i,i+k},A_{i+1,i+k},A_{i+1,i+k+1}\}=\{A_p,A_q,A_r\}(p=1,2,4,5,6,8;q=1,3,4,5,7,8;r=1,5)$ or $(p=1,2,4,5,6,8;q=2,6;r=2,3,4,6,7,8)$ or $(p=3,7;q=1,3,4,5,7,8;r=2,3,4,6,7,8)$, then $A_{i,i+k+1}$ can be any one matrix of $A_1,A_2,A_3,A_5,A_6,A_7$.
  \item[\textbf{Step 5.}]If $A_{i,i}= A_1$ (for all $i$), $A\leftarrow A+A^T$, else if $A_{i,i}= A_7$ (for all $i$), $A_{i,i}\leftarrow \frac{1}{2}A_{i,i}$ and $A\leftarrow A+A^T$; If $Q_{2m}^TAQ_{2m}$ is an adjacency matrix of a graph, output $A$; otherwise go to Step 1.
\end{itemize}

\noindent

\begin{exam}
 \rm{An adjacency matrix $A$ with the following form is randomly generated by the above algorithm corresponding to the given parameters, so we directly have $Q_k^TAQ_k$ is a (0,1)-matrix.}

\begin{eqnarray*}
\tiny{~~~~~~~~~~~A=\left[\begin{array}{cc|cc|cc|cc|cc|cc|cc|cc|cc}
  0&0&0&0&0&1&1&1&0&1&1&0&0&1&1&1&0&0\\ 0&0&0&0&1&0&0&0&0&1&0&1&0&1&0&0&0&0\\ \hline 0&0&0&0&1&1&0&0&0&0&0&1&1&1&0&1&0&1\\ 0&0&0&0&0&0&0&0&1&1&0&1&0&0&0&1&1&0\\ \hline 0&1&1&0&0&0&0&1&0&1&0&0&1&0&0&0&0&1\\ 1&0&1&0&0&0&1&0&0&1&1&1&1&0&0&0&1&0\\ \hline 1&0&0&0&0&1&0&0&0&0&0&1&0&0&1&0&0&0\\ 1&0&0&0&1&0&0&0&1&1&0&1&1&1&1&0&1&1\\ \hline 0&0&0&1&0&0&0&1&0&0&1&1&0&1&0&0&1&0\\ 1&1&0&1&1&1&0&1&0&0&0&0&0&1&0&0&0&1\\ \hline 1&0&0&0&0&1&0&0&1&0&0&0&0&0&1&0&1&1\\ 0&1&1&1&0&1&1&1&1&0&0&0&1&1&1&0&1&1\\ \hline 0&0&1&0&1&1&0&1&0&0&0&1&0&0&1&1&0&1\\ 1&1&1&0&0&0&0&1&1&1&0&1&0&0&0&0&0&1\\ \hline 1&0&0&0&0&0&1&1&0&0&1&1&1&0&0&0&1&1\\ 1&0&1&1&0&0&0&0&0&0&0&0&1&0&0&0&0&0\\ \hline 0&0&0&1&0&1&0&1&1&0&1&1&0&0&1&0&0&0\\ 0&0&1&0&1&0&0&1&0&1&1&1&1&1&1&0&0&0
\end{array}\right]},
\end{eqnarray*}
\begin{eqnarray*}
\tiny{Q_k^TAQ_k=\left[\begin{array}{cc|cc|cc|cc|cc|cc|cc|cc|cc}
    0&0&0&0&0&0&1&0&1&1&1&0&0&0&0&1&0&0 \\ 0&0&0&0&1&1&1&0&0&0&1&0&1&1&1&0&0&0 \\ \hline 0&0&0&0&0&1&0&0&1&0&1&1&1&0&0&1&0&1 \\ 0&0&0&0&1&0&0&0&1&0&0&0&1&0&1&0&1&0 \\ \hline 0&1&0&1&0&0&1&0&1&1&0&1&1&1&0&0&1&1 \\ 0&1&1&0&0&0&1&0&0&0&0&1&0&0&0&0&0&0 \\ \hline 1&1&0&0&1&1&0&0&1&0&0&0&0&1&1&0&0&1 \\ 0&0&0&0&0&0&0&0&1&0&1&1&0&1&0&1&0&1 \\ \hline 1&0&1&1&1&0&1&1&0&0&1&0&1&1&0&0&1&0 \\ 1&0&0&0&1&0&0&0&0&0&1&0&0&0&0&0&0&1 \\ \hline 1&1&1&0&0&0&0&1&1&1&0&0&0&1&0&0&1&1 \\ 0&0&1&0&1&1&0&1&0&0&0&0&0&1&1&1&1&1 \\ \hline 0&1&1&1&1&0&0&0&1&0&0&0&0&0&1&0&0&0 \\ 0&1&0&0&1&0&1&1&1&0&1&1&0&0&1&0&1&1 \\ \hline 0&1&0&1&0&0&1&0&0&0&0&1&1&1&0&0&0&1 \\ 1&0&1&0&0&0&0&1&0&0&0&1&0&0&0&0&0&1 \\ \hline 0&0&0&1&1&0&0&0&1&0&1&1&0&1&0&0&0&0 \\ 0&0&1&0&1&0&1&1&0&1&1&1&0&1&1&1&0&0
\end{array}\right]}.
\end{eqnarray*}
\end{exam}

According to the Algorithm 1, we can also construct families of infinite graphs such that  $Q_k^TAQ_k$ is a (0,1)-matrix.

\begin{exam}
\rm{ Let $A$ be the adjacency matrix of flower-graph $G$ (Fig.2) with vertex number $k=2m$ ($m\geq 3$ and odd number).  Dividing  $2m\times 2m$ matrix $A$ into 2 by 2 blocks and denote $A_{i,j}$ be the $(i,j)$-block of $A$. Then the upper triangular part of matrix $A$ is the following:

 The first diagonal block $A_{i,i}$ ($i=1,2,\cdots,m$) equals to $A_7$, the $\frac{m+1}{2}$-th diagonal block $A_{i,i+\frac{m+1}{2}-1}$ ($i=1,2,\cdots,\frac{m+1}{2}$) equals to $A_2$, the $\frac{m+3}{2}$-th diagonal block $A_{i,i+\frac{m+1}{2}}$ ($i=1,2,\cdots,\frac{m-1}{2}$) equals to $A_4$. So $Q_k^TAQ_k$ is (0,1)-matrix.}
\end{exam}

\begin{eqnarray*}
\tiny{
A=\left[\begin{array}{cc|cc|cc|cc|cc}
 0& 1& 0& 0& 1& 0& 1& 1& 0& 0\\
 1& 0& 0& 0& 1& 0& 0& 0& 0& 0\\\hline
 0& 0& 0& 1& 0& 0& 1& 0& 1& 1\\
 0& 0& 1& 0& 0& 0& 1& 0& 0& 0\\\hline
 1& 1& 0& 0& 0& 1& 0& 0& 1& 0\\
 0& 0& 0& 0& 1& 0& 0& 0& 1& 0\\\hline
 1& 0& 1& 1& 0& 0& 0& 1& 0& 0\\
 1& 0& 0& 0& 0& 0& 1& 0& 0& 0\\\hline
 0& 0& 1& 0& 1& 1& 0& 0& 0& 1\\
 0& 0& 1& 0& 0& 0& 0& 0& 1& 0
\end{array}\right]},
\tiny{Q_k^TAQ_k=\left[\begin{array}{cc|cc|cc|cc|cc}
 0& 1& 0& 0& 0& 0& 0& 1& 0& 0\\
 1& 0& 0& 0& 1& 1& 0& 1& 0& 0\\\hline
 0& 0& 0& 1& 0& 0& 0& 0& 0& 1\\
 0& 0& 1& 0& 0& 0& 1& 1& 0& 1\\\hline
 0& 1& 0& 0& 0& 1& 0& 0& 0& 0\\
 0& 1& 0& 0& 1& 0& 0& 0& 1& 1\\\hline
 0& 0& 0& 1& 0& 0& 0& 1& 0& 0\\
 1& 1& 0& 1& 0& 0& 1& 0& 0& 0\\\hline
 0& 0& 0& 0& 0& 1& 0& 0& 0& 1\\
 0& 0& 1& 1& 0& 1& 0& 0& 1& 0
\end{array}\right]}.
\end{eqnarray*}

\begin{figure}[htbp]
\centering
\subfloat[]{
\begin{minipage}[t]{0.3\textwidth}
   \centering
   \includegraphics[width=2.5cm,height=2.5cm]{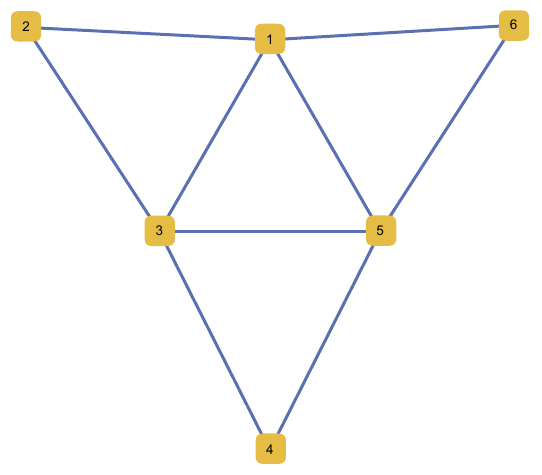}
\end{minipage}
}
\subfloat[]{
\begin{minipage}[t]{0.3\textwidth}
   \centering
   \includegraphics[width=2.5cm,height=2.5cm]{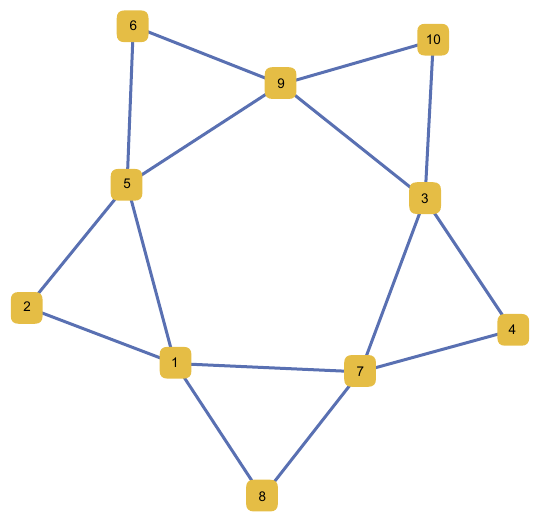}
\end{minipage}
}
\subfloat[]{
\begin{minipage}[t]{0.3\textwidth}
   \centering
   \includegraphics[width=2.5cm,height=2.5cm]{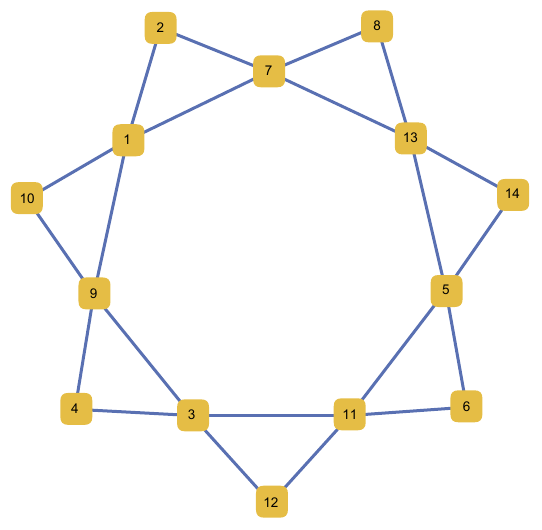}
\end{minipage}
}
\caption{flower-graph $G$ with vertices labeling}
\end{figure}

\begin{exam}
\rm{ Let $A$ be the adjacency matrix of sunflower-graph $G$ (Fig.3) with vertex number $k=2m$ ($m\geq 5$ and odd number).  Dividing  $2m\times 2m$ matrix $A$ into 2 by 2 blocks and denote $A_{i,j}$ be the $(i,j)$-block of $A$. Then the upper triangular part of matrix $A$ is the following:

The $\frac{m-1}{2}$-th diagonal block $A_{i,i+\frac{m-1}{2}-1}$ ($i=1,2,\cdots,\frac{m-1}{2}+2$) and the $\frac{m+3}{2}$-th diagonal block $A_{i,i+\frac{m-1}{2}+1}$ ($i=1,2,\cdots,\frac{m-1}{2}$) all equals to matrix $A_6$, the $\frac{m+1}{2}$-th diagonal block $A_{i,i+\frac{m-1}{2}}$ ($i=1,2,\cdots,\frac{m-1}{2}+1$) and the $\frac{m+5}{2}$-th diagonal block $A_{i,i+\frac{m-1}{2}+2}$ ($i=1,2,\cdots,\frac{m-1}{2}-1$) all equals to matrix $A_8$. So $Q_k^TAQ_k$ is (0,1)-matrix.}
\end{exam}

\begin{eqnarray*}
\tiny{
~~~~~~~~~~A=\left[\begin{array}{cc|cc|cc|cc|cc|cc|cc}
 0& 0& 0& 0& 0& 1& 0& 0& 0& 1& 0& 0& 0& 0\\
 0& 0& 0& 0& 0& 1& 1& 1& 0& 1& 1& 1& 0& 0\\\hline
 0& 0& 0& 0& 0& 0& 0& 1& 0& 0& 0& 1& 0& 0\\
 0& 0& 0& 0& 0& 0& 0& 1& 1& 1& 0& 1& 1& 1\\\hline
 0& 0& 0& 0& 0& 0& 0& 0& 0& 1& 0& 0& 0& 1\\
 1& 1& 0& 0& 0& 0& 0& 0& 0& 1& 1& 1& 0& 1\\\hline
 0& 1& 0& 0& 0& 0& 0& 0& 0& 0& 0& 1& 0& 0\\
 0& 1& 1& 1& 0& 0& 0& 0& 0& 0& 0& 1& 1& 1\\\hline
 0& 0& 0& 1& 0& 0& 0& 0& 0& 0& 0& 0& 0& 1\\
 1& 1& 0& 1& 1& 1& 0& 0& 0& 0& 0& 0& 0& 1\\\hline
 0& 1& 0& 0& 0& 1& 0& 0& 0& 0& 0& 0& 0& 0\\
 0& 1& 1& 1& 0& 1& 1& 1& 0& 0& 0& 0& 0& 0\\\hline
 0& 0& 0& 1& 0& 0& 0& 1& 0& 0& 0& 0& 0& 0\\
 0& 0& 0& 1& 1& 1& 0& 1& 1& 1& 0& 0& 0& 0
\end{array}\right]},
\end{eqnarray*}

\begin{eqnarray*}
\tiny{
Q_k^TAQ_k=\left[\begin{array}{cc|cc|cc|cc|cc|cc|cc}
 0& 0& 0& 0& 1& 1& 1& 0& 1& 1& 1& 0& 0& 0\\
 0& 0& 0& 0& 0& 0& 1& 0& 0& 0& 1& 0& 0& 0\\\hline
 0& 0& 0& 0& 0& 0& 1& 1& 1& 0& 1& 1& 1& 0\\
 0& 0& 0& 0& 0& 0& 0& 0& 1& 0& 0& 0& 1& 0\\\hline
 1& 0& 0& 0& 0& 0& 0& 0& 1& 1& 1& 0& 1& 1\\
 1& 0& 0& 0& 0& 0& 0& 0& 0& 0& 1& 0& 0& 0\\\hline
 1& 1& 1& 0& 0& 0& 0& 0& 0& 0& 1& 1& 1& 0\\
 0& 0& 1& 0& 0& 0& 0& 0& 0& 0& 0& 0& 1& 0\\\hline
 1& 0& 1& 1& 1& 0& 0& 0& 0& 0& 0& 0& 1& 1\\
 1& 0& 0& 0& 1& 0& 0& 0& 0& 0& 0& 0& 0& 0\\\hline
 1& 1& 1& 0& 1& 1& 1& 0& 0& 0& 0& 0& 0& 0\\
 0& 0& 1& 0& 0& 0& 1& 0& 0& 0& 0& 0& 0& 0\\\hline
 0& 0& 1& 1& 1& 0& 1& 1& 1& 0& 0& 0& 0& 0\\
 0& 0& 0& 0& 1& 0& 0& 0& 1& 0& 0& 0& 0& 0
\end{array}\right]}.
\end{eqnarray*}

\begin{figure}[htbp]
\centering
\subfloat[]{
\begin{minipage}[t]{0.3\textwidth}
   \centering
   \includegraphics[width=3cm,height=3cm]{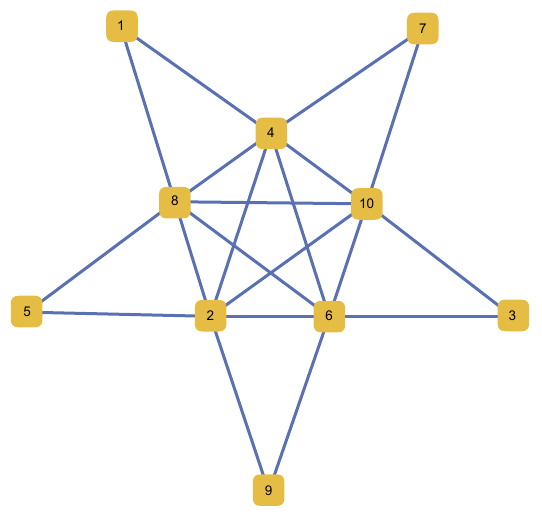}
\end{minipage}
}
\subfloat[]{
\begin{minipage}[t]{0.3\textwidth}
   \centering
   \includegraphics[width=3cm,height=3cm]{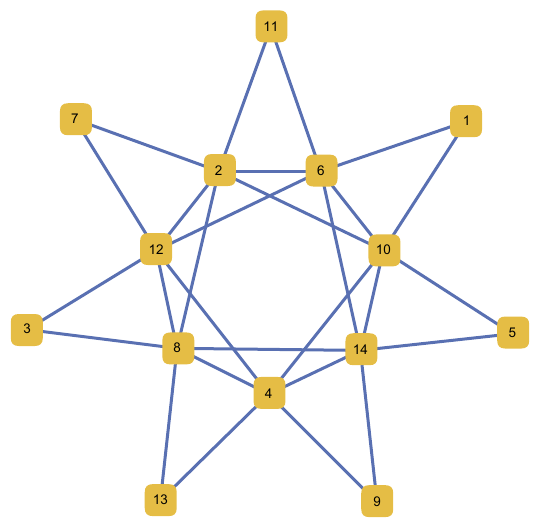}
\end{minipage}
}
\subfloat[]{
\begin{minipage}[t]{0.3\textwidth}
   \centering
   \includegraphics[width=3cm,height=3cm]{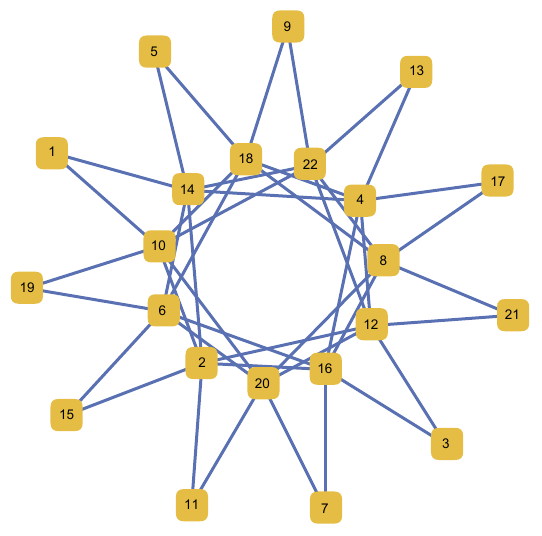}
\end{minipage}
}
\caption{sunflower-graph $G$ with vertices labeling}
\end{figure}

\begin{exam}
\rm{ Let $A$ be the adjacency matrix of graph $G$  with vertex number $n=2m$ ($m\geq 5$ and odd number).  Dividing  $2m\times 2m$ matrix $A$ into 2 by 2 blocks and denote $A_{i,j}$ be the $(i,j)$-block of $A$. Then the upper triangular part of matrix $A$ is the following:

The $\frac{m-1}{2}$-th diagonal block $A_{i,i+\frac{m-1}{2}-1}$ ($i=1,2,\cdots,\frac{m-1}{2}+2$) and the $\frac{m+5}{2}$-th diagonal block $A_{i,i+\frac{m-1}{2}+2}$ ($i=1,2,\cdots,\frac{m-1}{2}-1$) all equals to matrix $A_7$, the $\frac{m+1}{2}$-th diagonal block $A_{i,i+\frac{m-1}{2}}$ ($i=1,2,\cdots,\frac{m-1}{2}+1$) all equals to matrix $A_6$, the $\frac{m+3}{2}$-th diagonal block $A_{i,i+\frac{m-1}{2}+1}$ ($i=1,2,\cdots,\frac{m-1}{2}$) all equals to matrix$A_8$.}
\end{exam}

\begin{eqnarray*}
\tiny{
~~~~~~~~~~A=\left[\begin{array}{cc|cc|cc|cc|cc|cc|cc}
  0& 0& 0& 0& 0& 1& 0& 1& 0& 0& 0& 1& 0& 0\\
 0& 0& 0& 0& 1& 0& 0& 1& 1& 1& 1& 0& 0& 0\\\hline
 0& 0& 0& 0& 0& 0& 0& 1& 0& 1& 0& 0& 0& 1\\
 0& 0& 0& 0& 0& 0& 1& 0& 0& 1& 1& 1& 1& 0\\\hline
 0& 1& 0& 0& 0& 0& 0& 0& 0& 1& 0& 1& 0& 0\\
 1& 0& 0& 0& 0& 0& 0& 0& 1& 0& 0& 1& 1& 1\\\hline
 0& 0& 0& 1& 0& 0& 0& 0& 0& 0& 0& 1& 0& 1\\
 1& 1& 1& 0& 0& 0& 0& 0& 0& 0& 1& 0& 0& 1\\\hline
 0& 1& 0& 0& 0& 1& 0& 0& 0& 0& 0& 0& 0& 1\\
 0& 1& 1& 1& 1& 0& 0& 0& 0& 0& 0& 0& 1& 0\\\hline
 0& 1& 0& 1& 0& 0& 0& 1& 0& 0& 0& 0& 0& 0\\
 1& 0& 0& 1& 1& 1& 1& 0& 0& 0& 0& 0& 0& 0\\\hline
 0& 0& 0& 1& 0& 1& 0& 0& 0& 1& 0& 0& 0& 0\\
 0& 0& 1& 0& 0& 1& 1& 1& 1& 0& 0& 0& 0& 0
\end{array}\right]},
\end{eqnarray*}

\begin{eqnarray*}
\tiny{
Q_k^TAQ_k=\left[\begin{array}{cc|cc|cc|cc|cc|cc|cc}
  0& 0& 0& 0& 0& 1& 1& 1& 1& 0& 0& 1& 0& 0\\
 0& 0& 0& 0& 1& 0& 0& 0& 1& 0& 1& 0& 0& 0\\\hline
 0& 0& 0& 0& 0& 0& 0& 1& 1& 1& 1& 0& 0& 1\\
 0& 0& 0& 0& 0& 0& 1& 0& 0& 0& 1& 0& 1& 0\\\hline
 0& 1& 0& 0& 0& 0& 0& 0& 0& 1& 1& 1& 1& 0\\
 1& 0& 0& 0& 0& 0& 0& 0& 1& 0& 0& 0& 1& 0\\\hline
 1& 0& 0& 1& 0& 0& 0& 0& 0& 0& 0& 1& 1& 1\\
 1& 0& 1& 0& 0& 0& 0& 0& 0& 0& 1& 0& 0& 0\\\hline
 1& 1& 1& 0& 0& 1& 0& 0& 0& 0& 0& 0& 0& 1\\
 0& 0& 1& 0& 1& 0& 0& 0& 0& 0& 0& 0& 1& 0\\\hline
 0& 1& 1& 1& 1& 0& 0& 1& 0& 0& 0& 0& 0& 0\\
 1& 0& 0& 0& 1& 0& 1& 0& 0& 0& 0& 0& 0& 0\\\hline
 0& 0& 0& 1& 1& 1& 1& 0& 0& 1& 0& 0& 0& 0\\
 0& 0& 1& 0& 0& 0& 1& 0& 1& 0& 0& 0& 0& 0
\end{array}\right]}.
\end{eqnarray*}

\begin{rem}
 \rm{We can change all the first diagonal blocks $A_7$ to $A_1$ or $A_1$ to $A_7$. For non first diagonal blocks replace some matrices $A_1$ by $A_5$, or some matrices $A_2$ by $A_6$, or some matrices $A_3$ by $A_7$ or some matrices $A_4$ by $A_8$ randomly to get other more families of infinite graphs such that  $Q_k^TAQ_k$ is a (0,1)-matrix.}
\end{rem}

\noindent \textbf{3.2. $Q_{3m}$ has the form in Theorem \ref{Q1} case(ii)}

For 3 by 3  $(0,1)$-matrix, there are 512 matrices in total and only eight symmetric matrices $A_1,A_2,\cdots,A_8 $ of the them have zero diagonal elements. Here we just give $A_1,A_2,\cdots,A_8 $ and will not list 512 matrices all.

\begin{equation*}
\begin{aligned}
&\tiny{A_1=\left[\begin{array}{ccc}
 0& 0& 0\\
 0& 0& 0\\
 0& 0& 0
\end{array}\right]},\tiny{A_2=\left[\begin{array}{ccc}
 0& 1& 0\\
 1& 0& 0\\
 0& 0& 0
\end{array}\right]},\tiny{A_3=\left[\begin{array}{ccc}
 0& 0& 1\\
 0& 0& 0\\
 1& 0& 0
\end{array}\right]},\tiny{A_4=\left[\begin{array}{ccc}
 0& 0& 0\\
 0& 0& 1\\
 0& 1& 0
\end{array}\right]},\\
&\tiny{A_5=\left[\begin{array}{ccc}
 0& 1& 1\\
 1& 0& 1\\
 1& 1& 0
\end{array}\right]},\tiny{A_6=\left[\begin{array}{ccc}
 0& 0& 1\\
 0& 0& 1\\
 1& 1& 0
\end{array}\right]},\tiny{A_7=\left[\begin{array}{ccc}
 0& 1& 0\\
 1& 0& 1\\
 0& 1& 0
\end{array}\right]},\tiny{A_8=\left[\begin{array}{ccc}
 0& 1& 1\\
 1& 0& 0\\
 1& 0& 0
\end{array}\right]}.
\end{aligned}
\end{equation*}

Let $A$ be the adjacency matrix of graph $G$ with vertex number $k=3m$. Dividing  $3m\times 3m$ matrix $A$ into 3 by 3 blocks and denote $A_{i,j}$ be the $(i,j)$-th block of $A$. To generate to the first diagonal block of $A$ we have the following theorem.

\begin{them} \label{Q5}
Let $X_1=\left[\begin{array}{c}
J_2-2I_2\\
J_2
\end{array}\right]$ be a $2\times 1$ block matrix, $M=\left[\begin{array}{cc}
A_{i,i}&A_{i,i+1}\\
A_{i,i+1}^T&A_{i+1,i+1}
\end{array}\right]$, then $\frac{1}{9}X_1^TMX_1$ is a (0,1)-matrix if and only if one of the following is true:

\begin{enumerate}
  \item the first diagonal block $A_{1,1}=A_{2,2}=\cdots=A_{m,m}=A_1$ or $A_{1,1}=A_{2,2}=\cdots=A_{m,m}=A_5$. And the second diagonal block  $A_{i,i+1}$ ($i=1,2,\cdots,m-1$) can only be chosen from one of 56 matrices $B_1,B_2,\cdots,B_{56}$ (see Appendix).
  \item the first diagonal block $A_{i,i} (i=1,2,\cdots,m)$ can only choose any one of $A_2,A_3,A_4$ or all $A_{i,i} (i=1,2,\cdots,m)$ can only choose any one of $A_6,A_7,A_8$.
\end{enumerate}
\end{them}

\begin{proof}
For the first diagonal block $A_{i,i}$ and $A_{i+1,i+1}$ there are 8 possibilities which they can be chosen from $A_1, A_2,\cdots A_8$. For the second diagonal block $A_{i,i+1}$ there are 512 possibilities . So there are 32768 possibilities in total such that $\frac{1}{9}X_1^TMX_1$ is a (0,1)-matrix.  All these possibilities can be verified directly by computer. So the conclusion is proven.
\end{proof}

 From Theorem \ref{Q5}, we can see that it is too complicated  if we consider all 512 matrices. In order to simplify the algorithm, we only consider two special cases $A_{1,1}=A_{2,2}=\cdots=A_{m,m}=A_1$ or $A_{1,1}=A_{2,2}=\cdots=A_{m,m}=A_5$. Next we select eight (0,1)-matrices which will be used in the following algorithm, mark symbol matrix as $C_1,C_2,\cdots,C_8$ below.

\begin{equation*}
\begin{aligned}
&\tiny{C_1=\left[\begin{array}{ccc}
 0& 0& 0\\
 0& 0& 0\\
 0& 0& 0
\end{array}\right]},\tiny{C_2=\left[\begin{array}{ccc}
 1& 0& 0\\
 1& 0& 0\\
 1& 0& 0
\end{array}\right]},\tiny{C_3=\left[\begin{array}{ccc}
 1& 0& 0\\
 0& 1& 0\\
 0& 0& 1
\end{array}\right]},\tiny{C_4=\left[\begin{array}{ccc}
 1& 1& 1\\
 0& 0& 0\\
 0& 0& 0
\end{array}\right]},\\
&\tiny{C_5=\left[\begin{array}{ccc}
 1& 1& 1\\
 1& 1& 1\\
 1& 1& 1
\end{array}\right]},\tiny{C_6=\left[\begin{array}{ccc}
 0& 1& 1\\
 0& 1& 1\\
 0& 1& 1
\end{array}\right]},\tiny{C_7=\left[\begin{array}{ccc}
 0& 1& 1\\
 1& 0& 1\\
 1& 1& 0
\end{array}\right]},\tiny{C_8=\left[\begin{array}{ccc}
 0& 0& 0\\
 1& 1& 1\\
 1& 1& 1
\end{array}\right]}.
\end{aligned}
\end{equation*}

Partition the $3m\times 3m$ adjacency matrix $A$ into 3 by 3 blocks and denote by $A_{i,j}$ the $(i,j)$-th block of $A$. The adjacency matrix of graphs such that $Q_k^TAQ_k$ is a (0,1)-matrix is generated as follows:

\textbf{Algorithm 2:}

\begin{itemize}
  \item[\textbf{Step 1.}] Input parameters $m$, matrix $C_i(i=1,2,\cdots,8)$, regular rational orthogonal matrix $Q_{3m}$, $A\leftarrow 0$.
  \item[\textbf{Step 2.}] For $i=1$ to $m$, $A_{i,i}\leftarrow C_1$ (for all $i$) or $A_{i,i}\leftarrow C_7$ (for all $i$).
  \item[\textbf{Step 3.}] For $i=1$ to $m-1$, $A_{i,i+1}$ can choose any one of the six matrices $C_i (i=1,2,3,5,6,7)$.
  \item[\textbf{Step 4.}] For $k=1$ to $m-2$,

for $i=m-1-k$ to $1$,

\vspace{0.1cm}

If submatrix $\{A_{i,i+k},A_{i+1,i+k},A_{i+1,i+k+1}\}=\{C_p,C_q,C_r\}(p=1,2,4,5,6,8;q=1,3,4,5,7,8;r=6,7,8)$, then $A_{i,i+k+1}\leftarrow C_4$.

\vspace{0.1cm}

If submatrix $\{A_{i,i+k},A_{i+1,i+k},A_{i+1,i+k+1}\}=\{C_p,C_q,C_r\}(p=1,2,4,5,6,8;q=1,3,4,5,7,8;r=2,3,4)$, then $A_{i,i+k+1}\leftarrow C_8$.

\vspace{0.1cm}

If submatrix $\{A_{i,i+k},A_{i+1,i+k},A_{i+1,i+k+1}\}=\{C_p,C_q,C_r\}(p=1,2,4,5,6,8;q=1,3,4,5,7,8;r=1,5)$ or $\{A_{i,i+k},A_{i+1,i+k},A_{i+1,i+k+1}\}=\{C_p,C_2,C_r\}(p=1,2,4,5,6,8;r=6,7,8)$ or $\{A_{i,i+k},A_{i+1,i+k},A_{i+1,i+k+1}\}=\{C_p,C_6,C_r\}(p=1,2,4,5,6,8;r=2,3,4)$ or $\{A_{i,i+k},A_{i+1,i+k},A_{i+1,i+k+1}\}=\{C_7,C_q,C_r\}(q=1,3,4,5,7,8;r=6,7,8)$ or $\{A_{i,i+k},A_{i+1,i+k},A_{i+1,i+k+1}\}=\{C_3,C_q,C_r\}(q=1,3,4,5,7,8;r=2,3,4)$, then $A_{i,i+k+1}$ can be any one matrix of $C_1,C_2,C_3,C_5,C_6,C_7$.
  \item[\textbf{Step 5.}]If $A_{i,i}= C_1$ (for all $i$), then $A\leftarrow A+A^T$, else if $A_{i,i}= C_7$ (for all $i$), then $A_{i,i}\leftarrow \frac{1}{2}A_{i,i}$ and $A\leftarrow A+A^T$; If $Q_{3m}^TAQ_{3m}$ is an adjacency matrix of a graph, output $A$; otherwise go to Step 1.
\end{itemize}

\begin{exam}
\rm{Let $A$ be the adjacency matrix of graph $G$. Then, the matrix $Q_k^TAQ_k$ obtained is the adjacency matrix of graph $H$, which is cospectral with $G$ but nonisomorphic.}
\end{exam}
\begin{eqnarray*}
\tiny{
A=\left[\begin{array}{ccc|ccc|ccc|ccc}
 0& 0& 0& 0& 0& 0& 0& 0& 0& 1& 1& 1\\
 0& 0& 0& 0& 0& 0& 0& 0& 0& 1& 1& 1\\
 0& 0& 0& 0& 0& 0& 1& 1& 1& 1& 1& 1\\\hline
 0& 0& 0& 0& 0& 0& 1& 1& 0& 0& 0& 0\\
 0& 0& 0& 0& 0& 0& 1& 1& 0& 0& 0& 0\\
 0& 0& 0& 0& 0& 0& 1& 1& 0& 1& 1& 1\\\hline
 0& 0& 1& 1& 1& 1& 0& 0& 0& 1& 1& 0\\
 0& 0& 1& 1& 1& 1& 0& 0& 0& 1& 1& 0\\
 0& 0& 1& 0& 0& 0& 0& 0& 0& 1& 1& 0\\\hline
 1& 1& 1& 0& 0& 1& 1& 1& 1& 0& 0& 0\\
 1& 1& 1& 0& 0& 1& 1& 1& 1& 0& 0& 0\\
 1& 1& 1& 0& 0& 1& 0& 0& 0& 0& 0& 0
\end{array}\right]},
\tiny{
Q_k^TAQ_k=\left[\begin{array}{ccc|ccc|ccc|ccc}
 0&0&0&1&1&1&0&0&1&0&0&0\\
 0&0&0&1&1&1&0&0&1&0&0&0\\
 0&0&0&0&0&0&0&0&1&0&0&0\\\hline
 1&1&0&0&0&0&1&1&1&0&0&1\\
 1&1&0&0&0&0&1&1&1&0&0&1\\
 1&1&0&0&0&0&0&0&0&0&0&1\\\hline
 0&0&0&1&1&0&0&0&0&1&1&1\\
 0&0&0&1&1&0&0&0&0&1&1&1\\
 1&1&1&1&1&0&0&0&0&1&1&1\\\hline
 0&0&0&0&0&0&1&1&1&0&0&0\\
 0&0&0&0&0&0&1&1&1&0&0&0\\
 0&0&0&1&1&1&1&1&1&0&0&0
\end{array}\right]}.
\end{eqnarray*}

\begin{exam}
\rm{ Let $A$ be the adjacency matrix of graph $G$ with vertex number $k=3m$ ($m\geq 5$). By Algorithm 2, we have constructed infinite families of $A$ such that  $Q_k^TAQ_k$ is again (0,1)-matrix. Dividing  $3m\times 3m$ matrix $A$ into 3 by 3 blocks and denote $A_{i,j}$ be the $(i,j)$-block of $A$. Then the upper triangular part of matrix $A$ is the following:  the second diagonal block $A_{i,i+1}$ ($i=1,2,\cdots,m-1$) equals to $C_2$, the third diagonal block $A_{i,i+2}$ ($i=1,2,\cdots,m-2$) equals to $C_8$, $A_{1,m-1}=C_6$ and $A_{1,m}=C_3$.}
\end{exam}

\begin{eqnarray*}
\tiny{
~~~~~~~~~~~~A=\left[\begin{array}{ccc|ccc|ccc|ccc|ccc|ccc}
 0& 0& 0& 1& 0& 0& 0& 0& 0& 0& 0& 0& 0& 1& 1& 1& 0& 0\\
 0& 0& 0& 1& 0& 0& 1& 1& 1& 0& 0& 0& 0& 1& 1& 0& 1& 0\\
 0& 0& 0& 1& 0& 0& 1& 1& 1& 0& 0& 0& 0& 1& 1& 0& 0& 1\\\hline
 1& 1& 1& 0& 0& 0& 1& 0& 0& 0& 0& 0& 0& 0& 0& 0& 0& 0\\
 0& 0& 0& 0& 0& 0& 1& 0& 0& 1& 1& 1& 0& 0& 0& 0& 0& 0\\
 0& 0& 0& 0& 0& 0& 1& 0& 0& 1& 1& 1& 0& 0& 0& 0& 0& 0\\\hline
 0& 1& 1& 1& 1& 1& 0& 0& 0& 1& 0& 0& 0& 0& 0& 0& 0& 0\\
 0& 1& 1& 0& 0& 0& 0& 0& 0& 1& 0& 0& 1& 1& 1& 0& 0& 0\\
 0& 1& 1& 0& 0& 0& 0& 0& 0& 1& 0& 0& 1& 1& 1& 0& 0& 0\\\hline
 0& 0& 0& 0& 1& 1& 1& 1& 1& 0& 0& 0& 1& 0& 0& 0& 0& 0\\
 0& 0& 0& 0& 1& 1& 0& 0& 0& 0& 0& 0& 1& 0& 0& 1& 1& 1\\
 0& 0& 0& 0& 1& 1& 0& 0& 0& 0& 0& 0& 1& 0& 0& 1& 1& 1\\\hline
 0& 0& 0& 0& 0& 0& 0& 1& 1& 1& 1& 1& 0& 0& 0& 1& 0& 0\\
 1& 1& 1& 0& 0& 0& 0& 1& 1& 0& 0& 0& 0& 0& 0& 1& 0& 0\\
 1& 1& 1& 0& 0& 0& 0& 1& 1& 0& 0& 0& 0& 0& 0& 1& 0& 0\\\hline
 1& 0& 0& 0& 0& 0& 0& 0& 0& 0& 1& 1& 1& 1& 1& 0& 0& 0\\
 0& 1& 0& 0& 0& 0& 0& 0& 0& 0& 1& 1& 0& 0& 0& 0& 0& 0\\
 0& 0& 1& 0& 0& 0& 0& 0& 0& 0& 1& 1& 0& 0& 0& 0& 0& 0
\end{array}\right]},
\end{eqnarray*}

\begin{eqnarray*}
\tiny{
Q_k^TAQ_k=\left[\begin{array}{ccc|ccc|ccc|ccc|ccc|ccc}
 0& 0& 0& 1& 1& 1& 0& 1& 1& 0& 0& 0& 0& 0& 0& 1& 0& 0\\
 0& 0& 0& 0& 0& 0& 0& 1& 1& 0& 0& 0& 0& 0& 0& 0& 1& 0\\
 0& 0& 0& 0& 0& 0& 0& 1& 1& 0& 0& 0& 0& 0& 0& 0& 0& 1\\\hline
 1& 0& 0& 0& 0& 0& 1& 1& 1& 0& 1& 1& 0& 0& 0& 0& 0& 0\\
 1& 0& 0& 0& 0& 0& 0& 0& 0& 0& 1& 1& 0& 0& 0& 1& 1& 1\\
 1& 0& 0& 0& 0& 0& 0& 0& 0& 0& 1& 1& 0& 0& 0& 1& 1& 1\\\hline
 0& 0& 0& 1& 0& 0& 0& 0& 0& 1& 1& 1& 0& 1& 1& 0& 0& 0\\
 1& 1& 1& 1& 0& 0& 0& 0& 0& 0& 0& 0& 0& 1& 1& 0& 0& 0\\
 1& 1& 1& 1& 0& 0& 0& 0& 0& 0& 0& 0& 0& 1& 1& 0& 0& 0\\\hline
 0& 0& 0& 0& 0& 0& 1& 0& 0& 0& 0& 0& 1& 1& 1& 0& 1& 1\\
 0& 0& 0& 1& 1& 1& 1& 0& 0& 0& 0& 0& 0& 0& 0& 0& 1& 1\\
 0& 0& 0& 1& 1& 1& 1& 0& 0& 0& 0& 0& 0& 0& 0& 0& 1& 1\\\hline
 0& 0& 0& 0& 0& 0& 0& 0& 0& 1& 0& 0& 0& 0& 0& 1& 1& 1\\
 0& 0& 0& 0& 0& 0& 1& 1& 1& 1& 0& 0& 0& 0& 0& 0& 0& 0\\
 0& 0& 0& 0& 0& 0& 1& 1& 1& 1& 0& 0& 0& 0& 0& 0& 0& 0\\\hline
 1& 0& 0& 0& 1& 1& 0& 0& 0& 0& 0& 0& 1& 0& 0& 0& 0& 0\\
 0& 1& 0& 0& 1& 1& 0& 0& 0& 1& 1& 1& 1& 0& 0& 0& 0& 0\\
 0& 0& 1& 0& 1& 1& 0& 0& 0& 1& 1& 1& 1& 0& 0& 0& 0& 0
\end{array}\right]},
\end{eqnarray*}

\begin{rem}
\rm{We give the equivalent class about the matrix below (the corresponding complement matrix is also an equivalent class). A matrix in the above Algorithm 2 can be replaced by any of its corresponding equivalence class. Then we can get more families of infinite graphs such that  $Q_k^TAQ_k$ is a (0,1)-matrix.}
\begin{equation*}
    \begin{aligned}
    &\tiny{\left[\begin{array}{ccc}
 1& 0& 0\\
 1& 0& 0\\
 1& 0& 0
\end{array}\right]}\sim\tiny{\left[\begin{array}{ccc}
 0& 1& 0\\
 0& 1& 0\\
 0& 1& 0
\end{array}\right]}\sim\tiny{\left[\begin{array}{ccc}
 0& 0& 1\\
 0& 0& 1\\
 0& 0& 1
\end{array}\right]}\\
&\tiny{\left[\begin{array}{ccc}
 1& 0& 0\\
 0& 1& 0\\
 0& 0& 1
\end{array}\right]}\sim\tiny{\left[\begin{array}{ccc}
 1& 0& 0\\
 0& 0& 1\\
 0& 1& 0
\end{array}\right]}\sim\tiny{\left[\begin{array}{ccc}
 0& 1& 0\\
 1& 0& 0\\
 0& 0& 1
\end{array}\right]}\sim\tiny{\left[\begin{array}{ccc}
 0& 1& 0\\
 0& 0& 1\\
 1& 0& 0
\end{array}\right]}\sim\tiny{\left[\begin{array}{ccc}
 0& 0& 1\\
 1& 0& 0\\
 0& 1& 0
\end{array}\right]}\sim\tiny{\left[\begin{array}{ccc}
 0& 0& 1\\
 0& 1& 0\\
 1& 0& 0
\end{array}\right]}\\
&\tiny{\left[\begin{array}{ccc}
 1& 1& 1\\
 0& 0& 0\\
 0& 0& 0
\end{array}\right]}\sim\tiny{\left[\begin{array}{ccc}
 0& 0& 0\\
 1& 1& 1\\
 0& 0& 0
\end{array}\right]}\sim\tiny{\left[\begin{array}{ccc}
 0& 0& 0\\
 0& 0& 0\\
 1& 1& 1
\end{array}\right]}
\end{aligned}
\end{equation*}
\end{rem}

\section{Constructing cospectral graphs}

Based on the previous results, in this section, we shall give two new methods of constructing cospectral graphs by using regular rational orthogonal matrices with level 2 and level 3.

Let $Q_k$ ($k=2m$ or $3m$) be a regular rational orthogonal matrix of the form in Theorem~\ref{Q1}. Choose a graph $X$ such that $Q_k^TA(X)Q_k$ is a (0,1)-matrix using the methods in the previous sections. Let
\begin{eqnarray}\label{Y}
A(G)=\left[\begin{array}{cc}
  A(X) & C \\
  C^T & B
\end{array}\right],~ Q=\left[\begin{array}{cc}
  Q_k & O \\
  O & I_{n-k}
\end{array}\right].
\end{eqnarray}

Then we have
\begin{eqnarray}\label{LOL}
Q^TA(G)Q=\left[\begin{array}{cc}
  Q_k^TA(X)Q_k & Q_k^TC \\
  C^TQ_k & B
\end{array}\right].
\end{eqnarray}

In view of Eq.~(\ref{LOL}), in order for $Q^TA(G)Q$ to be the adjacency matrix of a graph $G^{'}$, we must ensure that $Q_k^TC$ is also a (0,1)-matrix. Next, we shall investigate the conditions, under which $Q_k^TC$ is a (0,1)-matrix.

\noindent \textbf{4.1. $Q_k$ has the form of case (i) in Theorem 2.1}

 Set $k=2m$. Let $G$ be the graph with adjacency matrix $A(G)$ as given in Eq.~\eqref{Y} and $X$ be the induced subgraph of $G$ with $V(X)=\{1,2,\ldots,2m-1,2m\}$. In paper \cite{LHM-WW}, the authors have given the edge-switching rules such that $Q^TAQ$ is a (0,1)-matrix.

\begin{them}\label{HH}\cite{LHM-WW}
Let $V(X)$ be partitioned into $m$ groups $\{1,2\}$,$\{3,4\}$,$\ldots$,$\{2m-1,2m\}$. Then $Q_k^TC$ is a (0,1)-matrix if and only if for any vertex $v\in V(G)\setminus V(X)$, $v$ is either adjacent to exactly one vertex in each group or adjacent to both vertices in the same group for some groups.
\end{them}

The details of the construction are as follows. Suppose that the vertex set $V(G)$ can be partitioned into two parts $V=V_1\cup V_2$. Let $X=G[V_1]$ be the subgraph induced by $V_1$.
Suppose the following conditions hold:
\begin{itemize}
\item  $X$ is a graph on $2m$ vertices such that $Q_{k}^TA(X)Q_{k}$ is a (0,1)-matrix, e.g., $X$ can be selected as the flower-graph in Example 2, Example 3 or Example 4.
\item $\forall v\in {V_2}$, $v$ is either adjacent to exactly one vertex in each group or adjacent to both vertices in the same group for some groups.
\end{itemize}

 Let $G^'$ be the graph obtained from $G$ as follows: the adjacency matrices induced by $V_1$ and $V_2$ in $G^'$ are $A(G^'[V_1])=Q_{k}^TA(X)Q_{k}$ and $G^'[V_2]=G[V_2]$, respectively. The rules for edge-switching between $V_1$ and $V_2$ can be described as follows:

\begin{itemize}
\item  $\forall v\in {V_2}$, suppose that $v$ is adjacent to exactly one vertex in each group $\{2i-1,2i\}$, for $i=1,2,\ldots,m$. Then switch all the edges between $\{v\}$ and $V_1$, i.e., delete the existing $m$ edges in $G$ and add the other $m$ non-edges to the new graph $G^'$;

\item $\forall v\in {V_2}$, if $v$ is adjacent to both vertices in the same group for some groups, then shift all the edges between $\{v\}$ and $V_1$ in the following way:
if $v$ is adjacent to both vertices $v_{2i+1}$ and $v_{2i+2}$ in $G$, then $G$ is adjacent to both vertices $v_{2i-1}$ and $v_{2i}$ in the new graph $G^'$, for $i=0,1,\ldots,m-1$ ($v_{2m+1}=v_1, v_{2m+2}=v_2$).
\end{itemize}

Then $G$ and $G^{'}$ are cospectral with cospectral complements.

\noindent
\begin{exam}
\rm{Let $X$ be the flower-graph given as in Fig.2. Then $Q_6^T A(X)Q_6$ is the adjacency matrix of another
flower-graph. Let $Q={\rm diag}(Q_6,I_3)$. Then $A^{'}=Q^TAQ$ is the adjacency matrix of the graph $G^{'}$. It is easy to see that the graphs $G$ and $G^{'}$ are non-isomorphic, since they have different degree sequences.
Here we just give the edges corresponding to matrix $C$ (see Fig.~4).}
\end{exam}

\begin{eqnarray*}
\tiny{
A=\left[\begin{array}{cccccc|ccc}
0&1&0&1&0&0&1&0&1\\
1&0&0&1&1&1&0&0&0\\
0&0&0&1&0&1&0&1&1\\
1&1&1&0&0&1&1&1&0\\
0&1&0&0&0&1&0&1&1\\
0&1&1&1&1&0&1&1&0\\\hline
1&0&0&1&0&1&0&1&0\\
0&0&1&1&1&1&1&0&1\\
1&0&1&0&1&0&0&1&0
\end{array}\right],Q^TAQ=\left[\begin{array}{cccccc|ccc}
0&1&1&1&1&0&0&1&0\\
1&0&0&0&1&0&1&1&1\\
1&0&0&1&1&1&1&1&0\\
1&0&1&0&0&0&0&1&1\\
1&1&1&0&0&1&1&0&0\\
0&0&1&0&1&0&0&0&1\\\hline
0&1&1&0&1&0&0&1&0\\
1&1&1&1&0&0&1&0&1\\
0&1&0&1&0&1&0&1&0
\end{array}\right]}.
\end{eqnarray*}
\begin{figure}[t]
\centering   
\subfloat[$G$]{
\begin{minipage}[t]{0.4\textwidth}
   \centering
   \includegraphics[width=6cm,height=2.2cm]{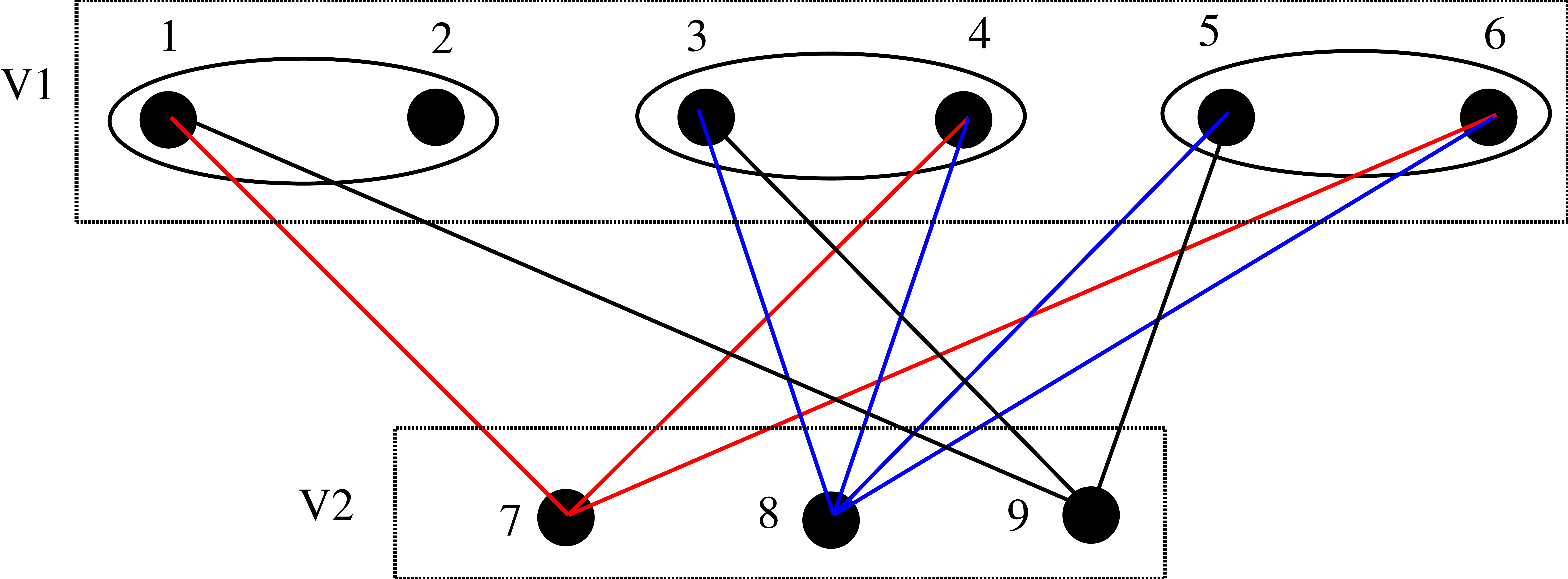}
\end{minipage}
}
\subfloat[$G^{'}$]{
\begin{minipage}[t]{0.4\textwidth}
   \centering
   \includegraphics[width=6cm,height=2.2cm]{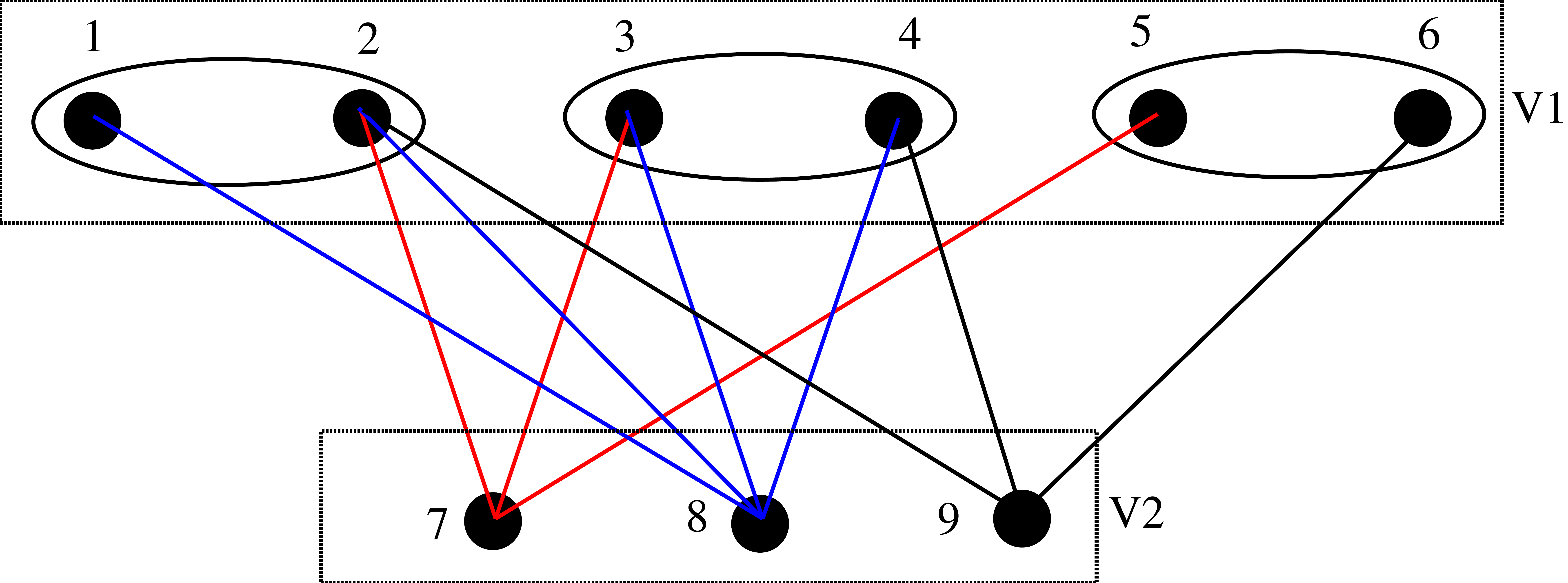}
\end{minipage}
}
\caption{Edge-switching for graph $G$}
\label{fig:ps}
\end{figure}

\noindent \textbf{4.2. $Q_k$ has the form of case (ii) in Theorem 2.1}

 Set $k=3m$. Let $G$ be the graph with adjacency matrix $A(G)$ as given in Eq.~\eqref{Y} and $X$ be the induced subgraph of $G$ with $V(X)=\{1,2,\ldots,3m-1,3m\}$.
 We have the following theorem.
\begin{them}\label{HH}
Let $V(X)$ be partitioned into $m$ groups $\{1,2,3\}$,$\{4,5,6\}$,$\ldots$,$\{3m-2,3m-1,3m\}$. Then $Q_k^TC$ is a (0,1)-matrix if and only if for any vertex $v\in V(G)\setminus V(X)$, $v$ is adjacent to three vertices in the same group for some groups.
\end{them}
\begin{proof}
Let $\alpha=(x_1,x_2,\ldots,x_{3m-1},{x_{3m}})^T$ be any column of matrix $C$. Then $Q_k^TC$ is a (0,1)-matrix if and only if $Q_k^T\alpha$ is a (0,1)-vector, i.e.,
 \begin{eqnarray}
-2x_{3i+1}+x_{3i+2}+x_{3i+3}+x_{3i+4}+x_{3i+5}+x_{3i+6}=0 ~{\rm or}~ 3, \\
x_{3i+1}-2x_{3i+2}+x_{3i+3}+x_{3i+4}+x_{3i+5}+x_{3i+6}=0 ~{\rm or}~ 3, \\
x_{3i+1}+x_{3i+2}-2x_{3i+3}+x_{3i+4}+x_{3i+5}+x_{3i+6}=0 ~{\rm or}~ 3,
\end{eqnarray}
for $i=0,1,2,\ldots,m-1$ ($x_{3m+4}=x_1,x_{3m+5}=x_2,x_{3m+6}=x_3$).

If $v$ is adjacent to three vertices in the same group for some groups, then $\forall i$, $(x_{3i+1},x_{3i+2},x_{3i+3},x_{3i+4},x_{3i+5},x_{3i+6})=(0,0,0,0,0,0)$ or $(1,1,1,0,0,0)$, or $(0,0,0,1,1,1)$ or $(1,1,1,1,1,1)$. Still, it is easy to see Eqs.~(3) , (4) and (5) hold.

To prove the necessity part, consider three equations $-2x+y+z+u+v+w=0$ or $3$  , $x+y-2z+u+v+w=0$ or $3$ and $x-2y+z+u+v+w=0$ or $3$ in variables $x,y,z,w\in{\{0,1\}}$. There are exactly four solutions to these three equations, i.e.,$(x,y,z,u,v,w)$ equals one of $(0,0,0,0,0,0)$ or $(1,1,1,0,0,0)$, or $(0,0,0,1,1,1)$ or $(1,1,1,1,1,1)$. Thus, for any $v\in V(G)\setminus V(X)$, $v$ is adjacent to three vertices in the same group for some groups. This completes the proof.
\end{proof}

Next, we summarize the details of the construction as follows. Suppose that the vertex set $V(G)$ can be partitioned into two parts $V=V_1\cup V_2$. Let $X=G[V_1]$ be the subgraph induced by $V_1$.
Suppose the following conditions hold:
\begin{itemize}
\item  $X$ is a graph on $3m$ vertices such that $Q_{k}^TA(X)Q_{k}$ is a (0,1)-matrix, e.g., $X$ can be selected as the graph in Example 6.
\item $\forall v\in {V_2}$, $v$ is adjacent to three vertices in the same group for some groups.
\end{itemize}

 Let $G^'$ be the graph obtained from $G$ as follows: the adjacency matrices induced by $V_1$ and $V_2$ in $G^'$ are $A(G^'[V_1])=Q_{k}^TA(X)Q_{k}$
 and $G^'[V_2]=G[V_2]$, respectively. The rules for edge-switching between $V_1$ and $V_2$ can be described as follows:

\begin{itemize}

\item $\forall v\in {V_2}$, if $v$ is adjacent to three vertices in the same group for some groups, then shift all the edges between $\{v\}$ and $V_1$ in the following way:
if $v$ is adjacent to three vertices $v_{3i+1}$, $v_{3i+2}$ and $v_{3i+3}$ in $G$, then $G$ is adjacent to both vertices  $v_{3i-2}$, $v_{3i-1}$ and $v_{3i}$ in the new graph $G^'$, for $i=0,1,\ldots,m-1$ ($v_{3m+1}=v_1, v_{3m+2}=v_2, v_{3m+3}=v_3$).

\end{itemize}
Then $G$ and $G^{'}$ are cospectral with cospectral complements.\\

\noindent
\begin{exam}
\rm{Let $X$ be the graph given as in Example 4. Then $Q_{15}^T A(X)Q_{15}$ is a (0,1)-matrix. Let $Q={\rm diag}(Q_{15},I_3)$. Then $A^{'}=Q^TAQ$ is the adjacency matrix of the graph $G^{'}$. By computers $G$ and $G^{'}$ are non-isomorphic. Here we just give the edges corresponding to matrix $C$(see Fig.~5).}
\end{exam}

\begin{eqnarray*}
	\tiny{
		~~~~~~~~~~A=\left[\begin{array}{ccccccccccccccc|ccc}
 0& 0& 0& 1& 0& 0& 0& 0& 0& 0& 1& 1& 1& 0& 0& 1& 0& 1\\
 0& 0& 0& 1& 0& 0& 1& 1& 1& 0& 1& 1& 0& 1& 0& 1& 0& 1\\
 0& 0& 0& 1& 0& 0& 1& 1& 1& 0& 1& 1& 0& 0& 1& 1& 0& 1\\
 1& 1& 1& 0& 0& 0& 1& 0& 0& 0& 0& 0& 0& 0& 0& 0& 1& 0\\
 0& 0& 0& 0& 0& 0& 1& 0& 0& 1& 1& 1& 0& 0& 0& 0& 1& 0\\
 0& 0& 0& 0& 0& 0& 1& 0& 0& 1& 1& 1& 0& 0& 0& 0& 1& 0\\
 0& 1& 1& 1& 1& 1& 0& 0& 0& 1& 0& 0& 0& 0& 0& 1& 0& 0\\
 0& 1& 1& 0& 0& 0& 0& 0& 0& 1& 0& 0& 1& 1& 1& 1& 0& 0\\
 0& 1& 1& 0& 0& 0& 0& 0& 0& 1& 0& 0& 1& 1& 1& 1& 0& 0\\
 0& 0& 0& 0& 1& 1& 1& 1& 1& 0& 0& 0& 1& 0& 0& 0& 1& 0\\
 1& 1& 1& 0& 1& 1& 0& 0& 0& 0& 0& 0& 1& 0& 0& 0& 1& 0\\
 1& 1& 1& 0& 1& 1& 0& 0& 0& 0& 0& 0& 1& 0& 0& 0& 1& 0\\
 1& 0& 0& 0& 0& 0& 0& 1& 1& 1& 1& 1& 0& 0& 0& 0& 1& 0\\
 0& 1& 0& 0& 0& 0& 0& 1& 1& 0& 0& 0& 0& 0& 0& 0& 1& 0\\
 0& 0& 1& 0& 0& 0& 0& 1& 1& 0& 0& 0& 0& 0& 0& 0& 1& 0\\\hline
 1& 1& 1& 0& 0& 0& 1& 1& 1& 0& 0& 0& 0& 0& 0& 0& 1& 0\\
 0& 0& 0& 1& 1& 1& 0& 0& 0& 1& 1& 1& 1& 1& 1& 1& 0& 1\\
 1& 1& 1& 0& 0& 0& 0& 0& 0& 0& 0& 0& 0& 0& 0& 0& 1& 0
		\end{array}\right]},
\end{eqnarray*}

\begin{eqnarray*}
	\tiny{
Q^TAQ=\left[\begin{array}{ccccccccccccccc|ccc}
 0& 0& 0& 1& 1& 1& 0& 1& 1& 0& 0& 0& 1& 0& 0& 0& 1& 0\\
 0& 0& 0& 0& 0& 0& 0& 1& 1& 0& 0& 0& 0& 1& 0& 0& 1& 0\\
 0& 0& 0& 0& 0& 0& 0& 1& 1& 0& 0& 0& 0& 0& 1& 0& 1& 0\\
 1& 0& 0& 0& 0& 0& 1& 1& 1& 0& 1& 1& 0& 0& 0& 1& 0& 0\\
 1& 0& 0& 0& 0& 0& 0& 0& 0& 0& 1& 1& 1& 1& 1& 1& 0& 0\\
 1& 0& 0& 0& 0& 0& 0& 0& 0& 0& 1& 1& 1& 1& 1& 1& 0& 0\\
 0& 0& 0& 1& 0& 0& 0& 0& 0& 1& 1& 1& 0& 1& 1& 0& 1& 0\\
 1& 1& 1& 1& 0& 0& 0& 0& 0& 0& 0& 0& 0& 1& 1& 0& 1& 0\\
 1& 1& 1& 1& 0& 0& 0& 0& 0& 0& 0& 0& 0& 1& 1& 0& 1& 0\\
 0& 0& 0& 0& 0& 0& 1& 0& 0& 0& 0& 0& 1& 1& 1& 0& 1& 0\\
 0& 0& 0& 1& 1& 1& 1& 0& 0& 0& 0& 0& 0& 0& 0& 0& 1& 0\\
 0& 0& 0& 1& 1& 1& 1& 0& 0& 0& 0& 0& 0& 0& 0& 0& 1& 0\\
 1& 0& 0& 0& 1& 1& 0& 0& 0& 1& 0& 0& 0& 0& 0& 1& 0& 1\\
 0& 1& 0& 0& 1& 1& 1& 1& 1& 1& 0& 0& 0& 0& 0& 1& 0& 1\\
 0& 0& 1& 0& 1& 1& 1& 1& 1& 1& 0& 0& 0& 0& 0& 1& 0& 1\\\hline
 0& 0& 0& 1& 1& 1& 0& 0& 0& 0& 0& 0& 1& 1& 1& 0& 1& 0\\
 1& 1& 1& 0& 0& 0& 1& 1& 1& 1& 1& 1& 0& 0& 0& 1& 0& 1\\
 0& 0& 0& 0& 0& 0& 0& 0& 0& 0& 0& 0& 1& 1& 1& 0& 1& 0
		\end{array}\right]}.
\end{eqnarray*}

\begin{figure}[t]
	\centering   
	\subfloat[$G$]{
		\begin{minipage}[t]{0.4\textwidth}
			\centering
			 \includegraphics[width=6.2cm,height=2.7cm]{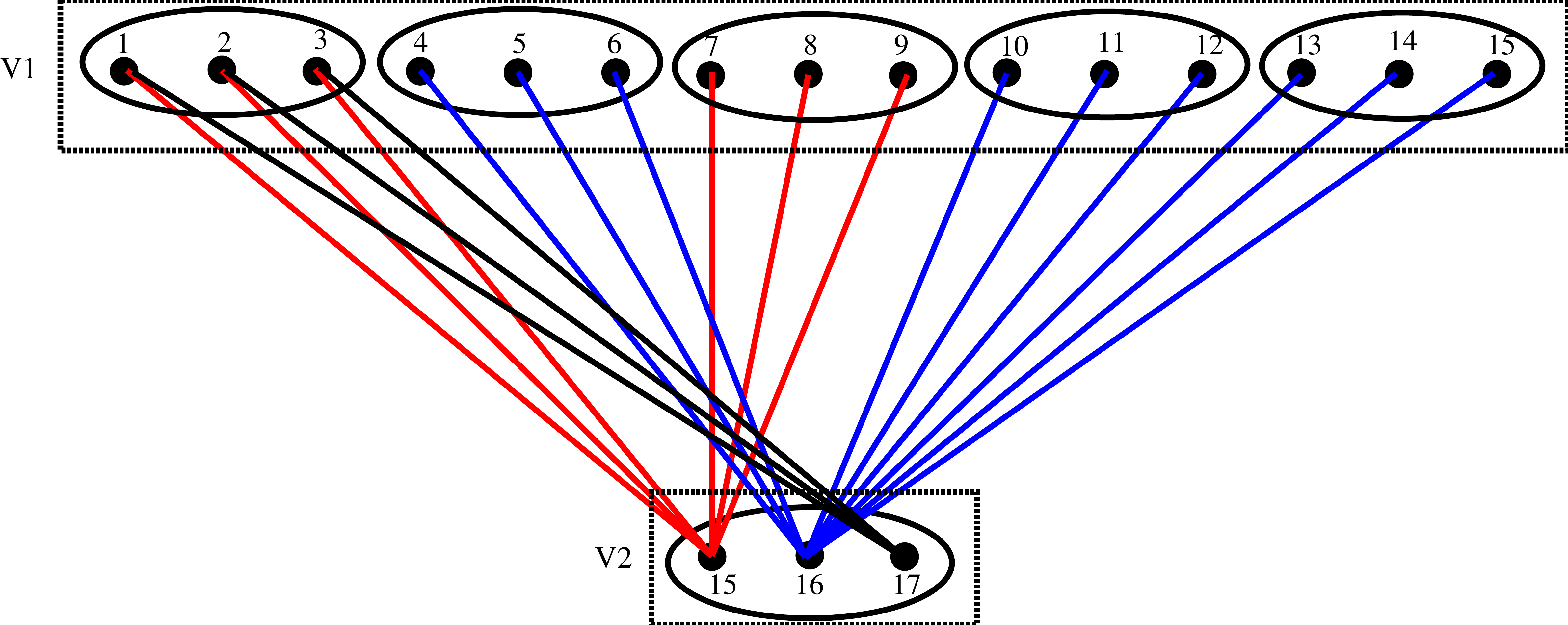}
		\end{minipage}
	}
	\subfloat[$G^{'}$]{
		\begin{minipage}[t]{0.4\textwidth}
			\centering
			\includegraphics[width=6.2cm,height=2.7cm]{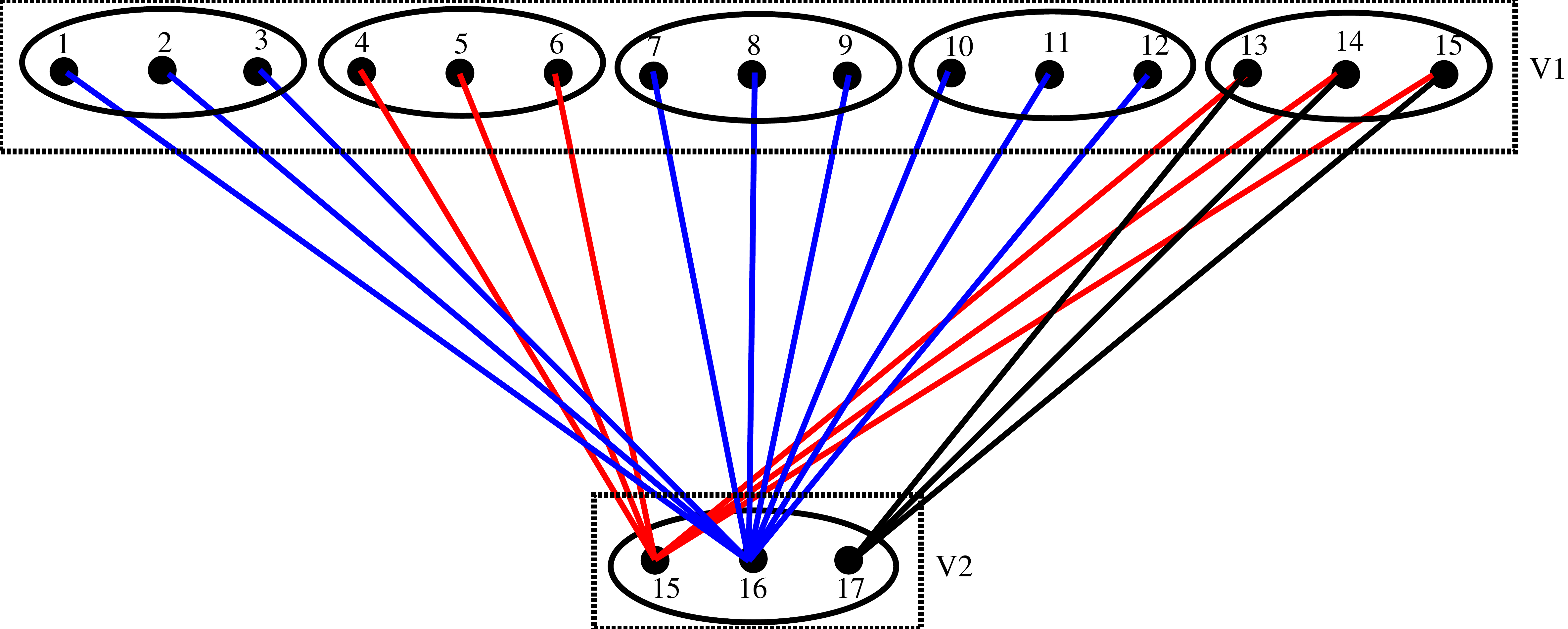}
		\end{minipage}
	}
	\caption{Edge-switching for graph $G$}
	\label{fig:ps}
\end{figure}

\section{Concluding remarks}

In this paper, we are mainly concerned with the problem of constructing cospectral graphs
through rational orthogonal matrix $Q$ of level two and three. We provide two algorithms and construct several infinite families of graphs $G$ with adjacency matrix $A$ such that $Q_k^TAQ_k$ is a (0,1)-matrix. Finally we give a new method of
constructing cospectral graphs by using regular rational orthogonal matrices with
level 2 and level 3. Our result gives some partial answers to Question 1, Question 2 and Question 3 in paper \cite{LHM-WW}. As a future work, it still would be interesting to investigate these problems.

\section{Appendix}

\begin{equation*}
\begin{aligned}
&\tiny{B_1=\left[\begin{array}{ccc}
   0& 0& 0\\
   0& 0& 0\\
   0& 0& 0
\end{array}\right]},\tiny{B_2=\left[\begin{array}{ccc}
0& 0& 1\\
   0& 0& 1\\
   0& 0& 1
\end{array}\right]},\tiny{B_3=\left[\begin{array}{ccc}
0& 0& 1\\
   0& 0& 1\\
   0& 1& 0
\end{array}\right]},\tiny{B_4=\left[\begin{array}{ccc}
   0& 0& 1\\
   0& 0& 1\\
   1& 0& 0
\end{array}\right]},\tiny{B_5=\left[\begin{array}{ccc}
   0& 0& 1\\
   0& 1& 0\\
   0& 0& 1
\end{array}\right],}\\
&\tiny{B_6=\left[\begin{array}{ccc}
 0& 0& 1\\
   0& 1& 0\\
   0& 1& 0
\end{array}\right]},\tiny{B_7=\left[\begin{array}{ccc}
 0& 0& 1\\
   0& 1& 0\\
   1& 0& 0
\end{array}\right]},\tiny{B_8=\left[\begin{array}{ccc}
   0& 0& 1\\
   1& 0& 0\\
   0& 0& 1
\end{array}\right]},\tiny{B_9=\left[\begin{array}{ccc}
  0& 0& 1\\
   1& 0& 0\\
   0& 1& 0
\end{array}\right]},\tiny{B_{10}=\left[\begin{array}{ccc}
0& 0& 1\\
   1& 0& 0\\
   1& 0& 0
\end{array}\right],}\\
&\tiny{B_{11}=\left[\begin{array}{ccc}
  0& 1& 0\\
   0& 0& 1\\
   0& 0& 1
\end{array}\right]},\tiny{B_{12}=\left[\begin{array}{ccc}
   0& 1& 0\\
   0& 0& 1\\
   0& 1& 0
\end{array}\right]},\tiny{B_{13}=\left[\begin{array}{ccc}
   0& 1& 0\\
   0& 0& 1\\
   1& 0& 0
\end{array}\right]},\tiny{B_{14}=\left[\begin{array}{ccc}
 0& 1& 0\\
   0& 1& 0\\
   0& 0& 1
\end{array}\right]},\tiny{B_{15}=\left[\begin{array}{ccc}
0& 1& 0\\
   0& 1& 0\\
   0& 1& 0
\end{array}\right],}\\
&\tiny{B_{16}=\left[\begin{array}{ccc}
  0& 1& 0\\
   0& 1& 0\\
   1& 0& 0
\end{array}\right]},\tiny{B_{17}=\left[\begin{array}{ccc}
   0& 1& 0\\
   1& 0& 0\\
   0& 0& 1
\end{array}\right]},\tiny{B_{18}=\left[\begin{array}{ccc}
   0& 1& 0\\
   1& 0& 0\\
   0& 1& 0
\end{array}\right]},\tiny{B_{19}=\left[\begin{array}{ccc}
 0& 1& 0\\
   1& 0& 0\\
   1& 0& 0
\end{array}\right]},\tiny{B_{20}=\left[\begin{array}{ccc}
   0& 1& 1\\
   0& 1& 1\\
   0& 1& 1
\end{array}\right],}\\
&\tiny{B_{21}=\left[\begin{array}{ccc}
   0& 1& 1\\
   0& 1& 1\\
   1& 0& 1
\end{array}\right]},\tiny{B_{22}=\left[\begin{array}{ccc}
   0& 1& 1\\
   0& 1& 1\\
   1& 1& 0
\end{array}\right]},\tiny{B_{23}=\left[\begin{array}{ccc}
   0& 1& 1\\
   1& 0& 1\\
   0& 1& 1
\end{array}\right]},\tiny{B_{24}=\left[\begin{array}{ccc}
 0& 1& 1\\
   1& 0& 1\\
   1& 0& 1
\end{array}\right]},\tiny{B_{25}=\left[\begin{array}{ccc}
  0& 1& 1\\
   1& 0& 1\\
   1& 1& 0
\end{array}\right],}\\
&\tiny{B_{26}=\left[\begin{array}{ccc}
0& 1& 1\\
   1& 1& 0\\
   0& 1& 1
\end{array}\right]},\tiny{B_{27}=\left[\begin{array}{ccc}
 0& 1& 1\\
   1& 1& 0\\
   1& 0& 1
\end{array}\right]},\tiny{B_{28}=\left[\begin{array}{ccc}
0& 1& 1\\
   1& 1& 0\\
   1& 1& 0
\end{array}\right]},\tiny{B_{29}=\left[\begin{array}{ccc}
    1& 0& 0\\
   0& 0& 1\\
   0& 0& 1
\end{array}\right]},\tiny{B_{30}=\left[\begin{array}{ccc}
   1& 0& 0\\
   0& 0& 1\\
   0& 1& 0
\end{array}\right],}\\
&\tiny{B_{31}=\left[\begin{array}{ccc}
1& 0& 0\\
   0& 0& 1\\
   1& 0& 0
\end{array}\right]},\tiny{B_{32}=\left[\begin{array}{ccc}
   1& 0& 0\\
   0& 1& 0\\
   0& 0& 1
\end{array}\right]},\tiny{B_{33}=\left[\begin{array}{ccc}
   1& 0& 0\\
   0& 1& 0\\
   0& 1& 0
\end{array}\right]},\tiny{B_{34}=\left[\begin{array}{ccc}
1& 0& 0\\
   0& 1& 0\\
   1& 0& 0
\end{array}\right]},\tiny{B_{35}=\left[\begin{array}{ccc}
   1& 0& 0\\
   1& 0& 0\\
   0& 0& 1
\end{array}\right],}\\
&\tiny{B_{36}=\left[\begin{array}{ccc}
   1& 0& 0\\
   1& 0& 0\\
   0& 1& 0
\end{array}\right]},\tiny{B_{37}=\left[\begin{array}{ccc}
 1& 0& 0\\
   1& 0& 0\\
   1& 0& 0
\end{array}\right]},\tiny{B_{38}=\left[\begin{array}{ccc}
  1& 0& 1\\
   0& 1& 1\\
   0& 1& 1
\end{array}\right]},\tiny{B_{39}=\left[\begin{array}{ccc}
1& 0& 1\\
   0& 1& 1\\
   1& 0& 1
\end{array}\right]},\tiny{B_{40}=\left[\begin{array}{ccc}
 1& 0& 1\\
   0& 1& 1\\
   1& 1& 0
\end{array}\right],}\\
&\tiny{B_{41}=\left[\begin{array}{ccc}
1& 0& 1\\
   1& 0& 1\\
   0& 1& 1
\end{array}\right]},\tiny{B_{42}=\left[\begin{array}{ccc}
1& 0& 1\\
   1& 0& 1\\
   1& 0& 1
\end{array}\right]},\tiny{B_{43}=\left[\begin{array}{ccc}
1& 0& 1\\
   1& 0& 1\\
   1& 1& 0
\end{array}\right]},\tiny{B_{44}=\left[\begin{array}{ccc}
 1& 0& 1\\
   1& 1& 0\\
   0& 1& 1
\end{array}\right]},\tiny{B_{45}=\left[\begin{array}{ccc}
 1& 0& 1\\
   1& 1& 0\\
   1& 0& 1
\end{array}\right],}\\
&\tiny{B_{46}=\left[\begin{array}{ccc}
 1& 0& 1\\
   1& 1& 0\\
   1& 1& 0
\end{array}\right]},\tiny{B_{47}=\left[\begin{array}{ccc}
 1& 1& 0\\
   0& 1& 1\\
   0& 1& 1
\end{array}\right]},\tiny{B_{48}=\left[\begin{array}{ccc}
   1& 1& 0\\
   0& 1& 1\\
   1& 0& 1
\end{array}\right]},\tiny{B_{49}=\left[\begin{array}{ccc}
1& 1& 0\\
   0& 1& 1\\
   1& 1& 0
\end{array}\right]},\tiny{B_{50}=\left[\begin{array}{ccc}
 1& 1& 0\\
   1& 0& 1\\
   0& 1& 1
\end{array}\right],}\\
&\tiny{B_{51}=\left[\begin{array}{ccc}
  1& 1& 0\\
   1& 0& 1\\
   1& 0& 1
\end{array}\right]},\tiny{B_{52}=\left[\begin{array}{ccc}
   1& 1& 0\\
   1& 0& 1\\
   1& 1& 0
\end{array}\right]},\tiny{B_{53}=\left[\begin{array}{ccc}
  1& 1& 0\\
   1& 1& 0\\
   0& 1& 1
\end{array}\right]},\tiny{B_{54}=\left[\begin{array}{ccc}
  1& 1& 0\\
   1& 1& 0\\
   1& 0& 1
\end{array}\right]},\tiny{B_{55}=\left[\begin{array}{ccc}
   1& 1& 0\\
   1& 1& 0\\
   1& 1& 0
\end{array}\right],}\\
&\tiny{B_{56}=\left[\begin{array}{ccc}
 1& 1& 1\\
   1& 1& 1\\
   1& 1& 1
\end{array}\right]}.
\end{aligned}
\end{equation*}

\begin{thebibliography}{99}
\bibitem{AA-WHH} A. Abiad, W.H. Haemers, Cospectral Graphs and Regular Orthogonal Matrices of
Level 2, Electron J. Comb., 2012(19): $\#$ P13.
\bibitem{BR} R. A. Brualdi, H.J. Ryser, Combinatorial Matrix
Theory, Cambridge University Press, 1991.
\bibitem{CDS} D.M. Cvetkovi\'{c}, M. Doob, H. Sachs, Spectra of Graphs, Academic Press, NewYork, 1982.
\bibitem{CDG-BDM}  C.D. Godsil, B.D. McKay, Constructing cospectral graphs, Aequationes Math., 1982(25): 257-268.
\bibitem{WHH} W.H. Haemers, X. Liu, Y. Zhang, Spectral Characterizations of lollipop graphs, Linear
Algebra Appl., 2008 (428):2415-2423.
\bibitem{YZJ-SCG-WW}Y.Z. Ji, S.C. Gong, W. Wang, Constructing cospectral bipartite graphs, Discrete Math., 2020(343): 112020.
\bibitem{LHM-WW}L.H. Mao, W. Wang, F.J. Liu, L.H. Qiu, Constructing cospectral graphs via regular rational
orthogonal matrices with level two, Discrete Math., 2023(346): 113156.
\bibitem{GRO}G.R. Omidi, B. K. Tajbakhsk, Starlike like trees are determined by their Laplacian spectrum,
Linear Algebra Appl., 2007(422): 654-658.
\bibitem{LHQ-YZJ-WW} L.H. Qiu, Y.Z. Ji, W. Wang, On a theorem of Godsil and McKay on the construction of
cospectral graphs, Linear Algebra Appl., 2020(603): 265-274.
\bibitem{AJS} A. J. Schwenk, Almost all trees are cospectral, in: F. Harary (Ed.), New Directions
in the Theory of Graphs, Academic Press, NewYork, 1973: 275-307.
\bibitem{SUN} T. Sunada, Riemannian coverings and isospectral manifolds, Ann. of Math., 1985(121): 169-186.
\bibitem{JHV-JJS} J. H. van Lint, J. J. Seidel, Equilateral point sets in elliptic geometry, Proc. Nederl. Akad. Wetenschappen A, 1966(69): 335-348.
\bibitem{WX1} W. Wang, C.X. Xu, An excluding algorithm for testing whether
a family of graphs are determined by their generalized spectra,
Linear Algebra Appl., 2006(418): 62-74.
\bibitem{WW} W. Wang, Generalized spectral characterization revisited, Electron J. Comb., 2013(20): $\#$ P4.
\bibitem{WW-LHQ-YLH} W. Wang, L.H. Qiu, Y.L. Hu, Cospectral graphs, GM-switching and regular
rational orthogonal matrices of level $p$, Linear Algebra Appl., 2019(563): 154-177.
\bibitem{JFW} J.F. Wang, Q.X. Huang, F. Belardo, E.M.L. Marzi, On the spectral characterization of $\infty$-graphs, Discrete Math., 2010(310): 1845-1855.

\end{thebibliography}
\end{document}